\documentclass[12pt]{article}

\usepackage{amsmath,amsthm,amssymb}
\usepackage{tikz}
\usepackage[colorlinks,
            linkcolor=red,
            anchorcolor=blue,
            citecolor=blue
            ]{hyperref}
\usepackage{pdfsync}
\usepackage[numbers,sort&compress]{natbib}
\allowdisplaybreaks

\theoremstyle{definition}

\theoremstyle{plain}
\newtheorem{theorem}{Theorem}[section]

\newtheorem{remark}{Remark}[section]
\newtheorem{lemma}{Lemma}[section]

\newtheorem{proposition}{Proposition}[section]

\newtheorem{corollary}{Corollary}[section]

\numberwithin{equation}{section}

\begin{document}
	\title{\bf Monotone quantities on $3$-manifolds with nonnegative scalar curvature }

 \author{Jiannan Chen$^{1}$\footnote{Email address: chenjiannan@email.ncu.edu.cn}, Haiping Fu$^{1}$\footnote{Email address: mathfu@126.com\,\,\    
 	}\\
 	{\small 1.Department of Mathematics,  Nanchang University,} {\small Nanchang 330031, Jiangxi, China.} 
    \\
}

\date{}

\maketitle
\begin{abstract}
In this paper, we derive monotone quantities for harmonic functions on asymptotically flat 3-manifolds with simple topology and nonnegative scalar curvature. These monotone quantities are constant
on spatial Schwarzschild manifolds outside rotationally symmetric spheres. To derive monotone quantities, our method is different from the ODE analysis in Xia-Yin-Zhou \cite{Xia} and Mazurowski-Yao \cite{Maz}, we follow the strategy developed in Miao \cite{Miao}. As applications, we recover and generalize some geometric inequalities and  mass-capacity inequalities in Miao \cite{Miao} and Oronzio \cite{Oronzio}. Furthermore, we obtain the integral identities for the mass-capacity ratio which is parallel to the results in Miao \cite{Miao}.   
\end{abstract}	

	

\section{\bf Introduction and main results}
Monotonicity quantities play a significant role in geometric analysis.  
Some classical results are given by Bishop–Gromov volume comparison theorem, Huisken monotonicity formula for the mean curvature flow \cite{Huisken1}, Perelman entropy formula for the Ricci flow \cite{Perelman}, and Geroch monotonicity of the Hawking mass along the inverse mean curvature flow \cite{Huisken2}. In \cite{Cold1,Cold2,Cold3}, Colding-Minicozzi established some monotonicity formulas along the level sets of the minimal positive Green function in nonparabolic Riemannian manifolds with nonnegative Ricci curvature. 

In recent years, there has been an increasing interest in investigating the level sets of harmonic functions over 3-manifolds with nonnegative scalar curvature. A pioneering work of Stern \cite{Stern} discovered intriguing analogy between the use of such level sets and the use of stable minimal surfaces instituted by Schoen and Yau \cite{SY1}. On asymptotically flat 3-manifolds, a new
proof of the positive mass theorem was obtained by Bray-Kazaras-Khuri-Stern\cite{Bray2},
which made use of harmonic functions that asymptotic to a linear coordinate function. Later, Munteanu-Wang \cite{Wang1} established a monotonicity formula for 3-manifolds with nonnegative scalar curvature and provided a sharp comparison results. Furthermore, Chodosh-Li \cite{Otis} generalized Munteanu-Wang's monotonicity formula to three-dimensional manifolds that do not have nonnegative scalar curvature, using this, they resolved the stable Bernstein problem in $ \mathbb{R}^{4} $. In \cite{Ag3}, Agostiniani,Mazzieri and Oronzio obtained another proof of the
Riemannian positive mass theorem through a different monotone quantity along level sets of the Green’s function on asymptotically flat 3-manifolds. In \cite{Miao}, Miao derived three monotone quantities 
\begin{equation}\label{A}
\mathcal{A}(t)=\frac{1}{1-t}\Big(8\pi-\frac{1}{1-t}\int_{\{u=t\}}H|\nabla u|\Big),
\end{equation}
\begin{equation}\label{B}
\mathcal{B}(t)=\frac{1}{1-t}\Big(4\pi-\frac{1}{(1-t)^{2}}\int_{\{u=t\}}|\nabla u|^{2}\Big),
\end{equation}
\begin{equation}\label{C}
\mathcal{D}(t)=4\pi(1-t)+\int_{\{u=t\}}H|\nabla u|-\frac{3}{1-t}\int_{\{u=t\}}|\nabla u|^{2}.
\end{equation}
along the level sets of a harmonic function $ u $ on 3-manifolds with simple topology and non-negative scalar curvature, by which he obtained several new proofs of the Positive Mass Theorem. Moreover, he established the following geometric inequalities
\begin{equation}\label{eq 1}
4\pi+\int_{\Sigma}H|\nabla u|\geq 3\int_{\Sigma}|\nabla u|^{2},
\end{equation}
\begin{equation}\label{eq 2}
4\pi-\int_{\Sigma}|\nabla u|^{2}\leq 4\pi \mathfrak{m}\mathfrak{c}_{\Sigma}^{-1},
\end{equation}
\begin{equation}\label{eq 3}
8\pi-\int_{\Sigma}H|\nabla u|\leq 12\pi\mathfrak{m}\mathfrak{c}_{\Sigma}^{-1},
\end{equation}
where $ u $ is the harmonic function on $ M $, $ \Sigma $ is the boundary of $ M $ and $ H $ is the mean curvature of $ \Sigma $. Moreover, equality in (\ref{eq 1}), (\ref{eq 2}) and (\ref{eq 3}) holds if and only if $ (M,g) $ is isometric to $ \mathbb{R}^{3} $ minus a ball. Inspired by Bray’s work \cite{Bray1}, by using the conformal method which preserves the nonnegativity of scalar curvature and the harmonicity of a function, Miao \cite{Miao} prmoted the geometric inequalities (\ref{eq 1}) and (\ref{eq 2}) to the following ones:
\begin{equation}\label{eq 4}
4\pi+k\int_{\Sigma}H|\nabla u|\geq k(4-k)\int_{\Sigma}|\nabla u|^{2},
\end{equation}
\begin{equation}\label{eq 5}
\frac{\mathfrak{m}}{\mathfrak{c}_{\Sigma}}\geq 2-\frac{1}{k}-\frac{k}{4\pi}\int_{\Sigma}|\nabla u|^{2}.
\end{equation}
Equality in both (\ref{eq 4}) and (\ref{eq 5}) holds if and only if $ (M,g) $ is isometric to a spatial Schwarzschild manifold outside a rotationally symmetric sphere.  Similar monotone quantities and geometric inequalities on capacitary functions have been also found by Oronzio \cite{Oronzio} independently. In \cite{Xia}, Xia-Yin-Zhou established general monotone quantities associated to p-capacitary functions whose model space is a spatial Schwarzschild manifold outside a rotationally symmetric sphere via an ODEs analysis. By using an ODE technique, Mazurowski-Yao \cite{Maz} obtained the similar monotone quantities independently. For more relevant results, refer to \cite{Fog1,Fog2,Mazz,Bray3,Hrisch1,Hrisch2,Mantou,Miao 0,Miao1}.

  In this paper, we derive and recover some monotone quantities for harmonic functions on asymptotically flat 3-manifolds with simple topology and nonnegative scalar curvature. These monotone quantities are constant
  on spatial Schwarzschild manifolds outside rotationally symmetric spheres. To establish monotone quantities, our method is different from the ODE analysis in  Xia-Yin-Zhou \cite{Xia} and Mazurowski-Yao \cite{Maz}, we follow the strategy developed in the work of Miao \cite{Miao}. In fact, in \cite{Miao}, Miao first obtained the monotone quantity (\ref{C}) then derived (\ref{A}) and (\ref{B}) from (\ref{C}). Inspired by this idea, we first use the monotone quantity (\ref{D t}) in \cite{Xia} to derive monotone quantity (\ref{P}) and recover two monotone quantities (\ref{Q}) and (\ref{F}) in \cite{Miao 0} and \cite{Xia}, respectively. Then, we apply the monotone quantity (\ref{F}) to derive the monotone quantity (\ref{G}). As applications, we recover and generalize some geometric inequalities and mass-capacity inequalities in \cite{Miao} and \cite{Oronzio}.  Furthermore, we obtain the integral identities for the mass-capacity ratio which is parallel to the results in \cite{Miao}.    In order to state our main results, we first introduce several preliminaries.

A 3-dimensional Riemannian manifold $ (M,g) $ (with or without boundary) is said to be one-ended asymptotically flat if there exists a compact set $ K\subset M $ such that $ M\setminus K $ is diffeomorphic to $ \mathbb{R}^{3} $ minus a ball, with respect to the standard coordinates on $ \mathbb{R}^{3} $, $ g $ satisfies
$$
g_{ij}=\delta_{ij}+O(|x|^{-\tau}),\,\,\,\partial g_{ij}=O(|x|^{-\tau-1}),\,\,\,\partial\partial g_{ij}=O(|x|^{-\tau-2})
$$
for some constant $ \tau>\frac{1}{2} $. 

The ADM mass of $ (M,g) $, which has been introduced in \cite{ADM}, is defined by 
$$
\mathfrak{m}=\lim\limits_{r\to \infty}\frac{1}{16\pi}\int\limits_{S_{r}}(\partial_{j}g_{ij}-\partial_{i}g_{ij})\frac{x^{i}}{|x|}d\sigma_{e},
$$
where $ S_{r}=\{|x|=r \} $ and $ d\sigma_{e} $ is the volume from induced from the Euclidean metric. The scalar curvature $ R_{g} $ of $ (M,g) $ is assumed to be integrable so that the ADM mass of $ (M,g) $ exists and it is independent of the asymptotically flat coordinate chart, see \cite{Bar}. 

On an asymptotically flat 3-manifold $ (M,g) $ with boundary $ \Sigma $, the capacity of $ \Sigma $ is given by 
$$
\mathfrak{c}_{\Sigma}=\inf\Bigl\{\frac{1}{4\pi}\int_{M}|\nabla f|^{2} \Big\},
$$
where the infimum is taken over all of locally Lischitz functions that equal $ 1 $ at $ \Sigma $ and tend to $ 0 $ at $ \infty $. Equivalently, we have 
$$
\mathfrak{c}_{\Sigma}=\frac{1}{4\pi}\int_{M}|\nabla \phi|^{2}=\frac{1}{4\pi}\int_{\Sigma}|\nabla \phi|, 
$$
where
$$
\Delta \phi=0,\,\,\,\phi |_{\Sigma}=1,\,\,\,\text{and }\phi\to 0\,\,\,\text{at } \infty.
$$

The spatial Schwarzschild manifold of mass $ \mathfrak{m} $ outside a rotationally symmetric sphere is the 3-dimensional Riemannian manifold with boundary given by the couple 
\begin{equation}\label{Sch}
\Big(\mathbb{R}^{3}\setminus B_{r_{0}}(0), (1+\frac{\mathfrak{m}}{2|x|})^{4}g_{eucl} \Big),
\end{equation}
which will be denoted by $ (M_{m,r_{0}},g_{m}) $.

Let $ (M,g) $ be an one-ended asymptotically flat manifold with boundary $ \Sigma $. Let $ u $ be the solution of the Dirichlet problem:
\begin{equation}\label{u}
\begin{cases}
\Delta u=0 &\text{in } M,\\
u=0, &\text{on } \Sigma,\\
u\to 1, &\text{at } \infty.
\end{cases}
\end{equation}
Given any $ t\in [0,1] $, let 
$$ 
\Sigma_{t}=\{x\in M| u(x)=t\} 
$$ 
denote the level set of $ u $.
In order to state the main results, we introduce the auxiliary function
$$
\omega(t)=1+(k-1)t,\,\,\, k>0.
$$
For convenience, we denote $ \omega=\omega(u) $. First we derive the following monotone quantity $ \mathcal{A}_{k}(t) $ for $ 0<k<4 $, which generalize the work in \cite{Miao}. Then we recover the monotone quantity $ \mathcal{B}_{k}(t) $ for $ k>0 $ in \cite{Miao 0} and some geometric inequalities in \cite{Miao}.
\begin{theorem}\label{th 1.1}
	Let $ (M,g) $ be a complete, orientable, one-ended asymptotically flat 3-manifold with compact, connnected boundary $ \Sigma $. Suppose that $ H_{2}(M,\Sigma)=0 $. Let $ u $ be the solution of (\ref{u}). If $ R_{g}\geq 0 $, then
	\\
	$ (i) $ For $ 0<k<4 $, the following quantity 
	\begin{equation}\label{P}
	 \mathcal{A}_{k}(t):=\frac{1}{1-t}\Big[g(t)\pi-\frac{k}{(1-t)(4\omega(t)-k)}\int_{\Sigma_{t}}H|\nabla u|\Big] 
    \end{equation}
	is monotone nondecreasing in $ t $, where $ g(t)=\frac{4\omega(t)-k}{k}-\frac{k}{4\omega(t)-k} $, $ g(1)=\frac{8}{3} $. Furthermore, we have
	\begin{equation}\label{lim A}
	\mathcal{A}_{k}(t)\leq 4\pi (\mathfrak{m}\mathfrak{c}_{\Sigma}^{-1}-\frac{2(k-1)}{k}).
	\end{equation}
	In particular, at $ \Sigma $,
	\begin{equation}\label{lim A0}
	(\frac{4-k}{k}-\frac{k}{4-k})\pi-\frac{k}{4-k}\int_{\Sigma}H|\nabla u|\leq 4\pi (\mathfrak{m}\mathfrak{c}_{\Sigma}^{-1}-\frac{2(k-1)}{k}),
	\end{equation}
	and equality holds if and only if  $ (M,g) $ is isometric to $ (M_{m,r_{0}},g_{m}) $ of mass $ \mathfrak{m}=2r_{0}(k-1) $ with $ r_{0}=\frac{\mathfrak{c}_{\Sigma}}{k} $.
	\\
	$ (ii) $ For $ k>0 $, the following quantity
	\begin{equation}\label{Q}
	\mathcal{B}_{k}(t):=\frac{\omega(t)}{1-t}\Big[4\pi-\frac{k^{2}}{(1-t)^{2}\omega(t)^{2}}\int_{\Sigma_{t}}|\nabla u|^{2} \Big]
	\end{equation}
	is monotone nondecreasing in $ t $. Furthermore, we have
	\begin{equation}\label{lim B}
	\mathcal{B}_{k}(t)\leq 2\pi (\mathfrak{m}\mathfrak{c}_{\Sigma}^{-1}-1).
	\end{equation}
	In particular, at $ \Sigma $,
	\begin{equation}\label{lim B0}
	4\pi-k^{2}\int_{\Sigma}|\nabla u|^{2}\leq 4\pi k (\mathfrak{m}\mathfrak{c}_{\Sigma}^{-1}-\frac{2(k-1)}{k}),
	\end{equation}
	and equality holds if and only if  $ (M,g) $ is isometric to $ (M_{m,r_{0}},g_{m}) $ of mass $ \mathfrak{m}=2r_{0}(k-1) $ with $ r_{0}=\frac{\mathfrak{c}_{\Sigma}}{k} $.
\end{theorem}
\begin{remark}
In \cite{Miao}, Miao derived the monotone quantity $ \mathcal{A}_{k}(t) $ and $ \mathcal{B}_{k}(t) $ for $ k=1 $. In \cite{Miao 0}, Miao derived the monotone quantity $ \mathcal{B}_{k}(t) $ for $ k>0 $.
\end{remark}
The following monotone quantitity $ \mathcal{G}_{k}(t) $ was first derived in Munteanu-Wang \cite{Wang1} for harmonic functions with $ k=1 $. Then, $ \mathcal{G}_{k}(t) $ was investigated in \cite{Chan,Hrisch1} for p-harmonic functions with $ p\in (1,3) $ and $ k=1 $. In \cite{Oronzio}, Oronzio considered $ \mathcal{G}_{k}(t) $ for harmonic functions with $ k=2 $. Here, we establish the monotonicity of $ \mathcal{G}_{k}(t) $ for harmonic functions with $ 0<k\leq2 $, which generalize the work in \cite{Wang1} and \cite{Oronzio}.
\begin{theorem}\label{th 1.2}
Let $ (M,g) $ be a complete, orientable, one-ended asymptotically flat 3-manifold with compact, connnected boundary $ \Sigma $ and with nonnegative scalar curvature. Suppose that $ H_{2}(M,\Sigma)=0 $. Let $ u $ be the solution of (\ref{u}) and let $ 0<k\leq 2 $. If there exists 
\begin{equation}\label{a condi}
\alpha\in (-\frac{1}{k\mathfrak{c}_{\Sigma}},\frac{1}{k\mathfrak{c}_{\Sigma}}] \end{equation}
such that 
\begin{equation}\label{H condition0}
H\leq \alpha(1-k^{2}\mathfrak{c}_{\Sigma}|\nabla u|)\,\,\,\text{on }\Sigma,
\end{equation}
then the following quantity
\begin{equation}\label{G}
\mathcal{G}_{k}(t)=-\frac{4\pi(1-t)}{k^{2}\omega(t)}+\frac{1}{(1-t)\omega(t)^{3}}\int_{\Sigma_{t}}|\nabla u|^{2}
\end{equation}
is monotone nondecreasing in $ t $. Moreover, we have
\begin{equation}\label{boundary}
\int_{\Sigma}|\nabla u|^{2}\leq \frac{4\pi}{k^{2}},
\end{equation}
and equality holds if and only if  $ (M,g) $ is isometric to $ (M_{m,r_{0}},g_{m}) $ of mass $ \mathfrak{m}=2r_{0}(k-1) $ with $ r_{0}=\frac{\mathfrak{c}_{\Sigma}}{k} $.
\end{theorem}
Applying Theorem~\ref{th 1.1}, we obtain (\ref{boundary}) under a weaker condition, and then mass-capacity inequality and area-capacity inequality, which extend the work in \cite{Oronzio}.
\begin{theorem}\label{th 1.3}
	Let $ (M,g) $ be a complete, orientable, one-ended asymptotically flat 3-manifold with $ R_{g}\geq 0 $ and with compact, connnected boundary $ \Sigma $. Suppose that $ H_{2}(M,\Sigma)=0 $. Let $ u $ be the solution of (\ref{u}) and let $ 0<k\leq 2 $. If there exists 
\begin{equation}
\alpha\in (-\frac{1}{k\mathfrak{c}_{\Sigma}},+\infty) 
\end{equation}
such that
\begin{equation}\label{condi H0}
H\leq \alpha(1-k^{2}\mathfrak{c}_{\Sigma}|\nabla u|)\,\,\,\,\text{on } \Sigma,
\end{equation}
	then we have the following inequalities
	$$
	\mathfrak{m}\geq \frac{2(k-1)\mathfrak{c}_{\Sigma}}{k},\,\,\,\,\sqrt{\frac{|\Sigma|}{4\pi k^{2}}}\geq \mathfrak{c}_{\Sigma},\,\,\,\,\int_{\Sigma}|\nabla u|^{2}\leq \frac{4\pi}{k^{2}},
	$$
and equality in each of the above inequalities holds if and only if $ (M,g) $ is isometric to $ (M_{m,r_{0}},g_{m}) $ of mass $ \mathfrak{m}=2r_{0}(k-1) $ with $ r_{0}=\frac{\mathfrak{c}_{\Sigma}}{k} $.
\end{theorem}
\begin{remark}
	In \cite{Oronzio}, $ k=2 $, $\alpha\in (-\frac{1}{2\mathfrak{c}_{\Sigma}},\frac{1}{2\mathfrak{c}_{\Sigma}}]$ in the mean curvature condition (\ref{condi H0}).
\end{remark}
In \cite{Bray2}, Bray–Kazaras–Khuri–Stern found an integral identity for the mass of anasymptotically flat manifold via using the harmonic function that asymptotic to a linear coordinate function. In \cite{Miao}, Miao derived mass identities with a harmonic
function $ u $ that equals $ 0 $ at the boundary and is asymptotic to $ 1 $ at $ \infty $, and he introduced a symmetric $ (0,2) $ tensor
$$
\Phi_{u}=\frac{|\nabla u|^{2}}{1-u}g-\frac{3du\otimes du}{1-u}.
$$
to explore the mass identities.
Motivated by the above results, we establish the following mass identites via a harmonic
function $ u $ that equals $ 0 $ at the boundary and is asymptotic to $ 1 $ at $ \infty $, and we introduce the following symmetric tensor
\begin{equation}\label{Psi}
\Psi_{u}=\frac{(2\omega-k)|\nabla u|^{2}}{(1-u)\omega}g-\frac{3(2\omega-k)}{(1-u)\omega}du\otimes du.
\end{equation}
to investigate the mass identities. We have the following result.
\begin{theorem}\label{th 1.6}
	Let $ (M,g) $ be a complete, orientable, one-ended asymptotically flat 3-manifold with compact, connnected boundary $ \Sigma $. Suppose that $ H_{2}(M,\Sigma)=0 $. Let $ u $ be the solution of (\ref{u}) and $ \Psi_{u} $ be defined as in (\ref{Psi}). Then
	\begin{equation}\label{int 1}
	\begin{aligned}
	&4\pi \Big(\mathfrak{m}\mathfrak{c}_{\Sigma}^{-1}-\frac{2(k-1)}{k}\Big)-\Big(\frac{4\pi}{k}-k\int_{\Sigma}|\nabla u|^{2}\Big)\\
	&\geq \frac{k}{4}\int_{M}\frac{1}{\omega^{2}}\Big[\frac{2\omega-k}{(1-u)^{2}}+k-2\Big]\Big[\frac{|\nabla^{2}u-\Psi_{u}|^{2}}{|\nabla u|}+R_{g}|\nabla u|\Big],
	\end{aligned}
	\end{equation}
	and
	\begin{equation}\label{int 2}
	\begin{aligned}
	&8\pi\Big(\mathfrak{m}\mathfrak{c}_{\Sigma}^{-1}-\frac{2(k-1)}{k}\Big)-\Big(\frac{4\pi}{k}+(4-3k)\int_{\Sigma}|\nabla u|^{2}-\int_{\Sigma}|\nabla u|H\Big)\\
	&\geq \int_{M}\frac{1}{2(1-u)^{2}}\Big[\frac{|\nabla^{2}u-\Psi_{u}|^{2}}{|\nabla u|}+R_{g}|\nabla u|\Big].
	\end{aligned}
	\end{equation}
\end{theorem}
This paper is organized as follows. In Section 2, we present some preliminary results. In Section 3, we derive monotone quantities and some geometric inequalities along the level sets. In Section 4, we obtain some mass-capacity inequalities. In Section 5, we establish the integral identities for the mass-capacity ratio. 
\section{\bf Preliminaries}
Let $ (M,g) $ be an one-ended asymptotically flat manifold with boundary $ \Sigma $. Let $ u $ be the solution of the Dirichlet problem (\ref{u}). By the maximum principle, $ Int(M)=\{0<u<1\} $. It follows that $ \Sigma=\{u=0\} $ and from the Hopf lemma that $ 0 $ is a regular value of $ u $.

It is known that in a asymptotically flat chart $ (x^{1},x^{2},x^{3}) $ of order $ \tau $, $\tau \in (\frac{1}{2},1)$, one has (see \cite{Mantou} for instance)
$$
u=1-\frac{\mathfrak{c}_{\Sigma}}{|x|}+O_{2}(|x|^{-1-\tau}).
$$
Here $ \mathfrak{c}_{\Sigma}>0 $ is a positive constant equal to the capacity of $ \Sigma $ in $ (M,g) $.

Given any $ t\in [0,1] $, recall that the level set of $ u $ is
$$ 
\Sigma_{t}=\{x\in M| u(x)=t\}. 
$$ 
Now, we collect some results about $ u $ and $ \Sigma_{t} $.  
\begin{lemma}\cite{Miao}\label{funda u}
Let $ u $ be the soulution of (\ref{u}). Suppose that $ t $ is a regular value of $ u $. Then, as $ t\to 1 $, we have the following
\\
$ (i) $ $ |\Sigma_{t}|=4\pi \mathfrak{c}_{\Sigma}^{2}(1-t)^{-2}+O((1-t)^{\tau-2}) $, where $ |\Sigma_{t}| $ is the area of $ \Sigma_{t} $ in $ (M,g) $;
\\
$ (ii) $ $ |\overset{\circ}{\mathbb{II}}|=O(|x|^{-1-\tau}) $, where $ |\overset{\circ}{\mathbb{II}}| $ denotes the traceless part of the second fundamental form $ \mathbb{II} $ of $ \Sigma_{t} $;
\\
$ (iii) $ $ K=|x|^{-2}+O(|x|^{-2-\tau}) $, where $ K $ is the Gauss curvature of $ \Sigma_{t} $;
\\
$ (iv) $ $\frac{1}{(1-t)}\int_{\Sigma_{t}}H|\nabla u|=8\pi+O((1-t)^{\tau})$, $\frac{1}{(1-t)^{2}}\int_{\Sigma_{t}}|\nabla u|^{2}=4\pi+O((1-t)^{\tau})$.
\end{lemma}
\begin{lemma}\cite{Miao}\label{deriva}
Let $ u $ be the soulution of (\ref{u}). Then, at a regular value $ t $, we have 
\begin{equation}
\begin{aligned}
&\Big(\int_{\Sigma_{t}}|\nabla u|^{2}\Big)^{\prime}=-\int_{\Sigma_{t}}|\nabla u|H,\\
&\Big(\int_{\Sigma_{t}}|\nabla u|H\Big)^{\prime}=\int_{\Sigma_{t}}\Big[-\frac{|\nabla^{\Sigma_{t}}|\nabla u||^{2}}{|\nabla u|^{2}}-\frac{1}{2}|\overset{\circ}{\mathbb{II}}|^{2}-\frac{3}{4}H^{2}+\frac{1}{2}R_{\Sigma_{t}}-\frac{1}{2}R_{g} \Big].
\end{aligned}
\end{equation}
\end{lemma}
To establish the monotonicity of $ \mathcal{D}_{k}(t) $, $ \mathcal{A}_{k}(t) $ and $ \mathcal{B}_{k}(t) $, we need the following regularization technique. The following argument follows from Miao \cite{Miao} and we extend the Miao's argument \cite{Miao} to the general case.
\begin{lemma}\label{f(u)1}
Let $ u $ be a harmonic function on a compact Riemannian manifold $ (\Omega,g) $ with boundary $ \partial \Omega $. Suppose $ f $ is a smooth positive function on $ \Omega $. Then
\begin{equation}
\int_{\partial \Omega}f(u)|\nabla u|\frac{\partial u}{\partial \zeta}=\int_{\Omega}f^{\prime}(u)|\nabla u|^{3}+\int_{\{|\nabla u|\neq 0\}} f(u)\frac{\nabla^{2}u(\nabla u,\nabla u)}{|\nabla u|}.
\end{equation}
Here $ \zeta $ denotes the unit normal to $ \partial \Omega $ pointing out of $ \Omega $.
\end{lemma}
\begin{proof}
Given any constant $ \epsilon>0 $, we have
$$
div\Big(f(u)\sqrt{|\nabla u|^{2}+\epsilon}\nabla u\Big)
=f^{\prime}(u)\sqrt{|\nabla u|^{2}+\epsilon}|\nabla u|^{2}+f(u)\frac{\nabla^{2}u(\nabla u,\nabla u)}{\sqrt{|\nabla u|^{2}+\epsilon}}.
$$
Therefore,
$$
\begin{aligned}
\int_{\partial \Omega}f(u)\sqrt{|\nabla u|^{2}+\epsilon}\frac{\partial u}{\partial \zeta}=&\int_{\Omega}f^{\prime}(u)\sqrt{|\nabla u|^{2}+\epsilon}|\nabla u|^{2}+\int_{\Omega}f(u)\frac{\nabla^{2}u(\nabla u,\nabla u)}{\sqrt{|\nabla u|^{2}+\epsilon}}\\
=&\int_{\Omega}f^{\prime}(u)\sqrt{|\nabla u|^{2}+\epsilon}|\nabla u|^{2}\\
&+\int_{\{|\nabla u|\neq 0\}}f(u)\nabla^{2}u\Big(\frac{\nabla u}{|\nabla u|},\frac{\nabla u}{|\nabla u|}\Big)\frac{|\nabla u|^{2}}{\sqrt{|\nabla u|^{2}+\epsilon}}.
\end{aligned}
$$
Taking $ \epsilon\to 0 $ gives the desired result.
\end{proof}
\begin{lemma}\label{f(u)2}
Let $ u $ be a harmonic function on a compact Riemannian manifold $ (\Omega,g) $ with boundary $ \partial \Omega $. Suppose $ f $ is a smooth positive function on $ \Omega $ and $ u $ equals a constant on each connected component of $ \partial \Omega $. Then
\begin{equation}\label{key}
\int_{\partial \Omega}f(u)|\nabla u|H\leq \int_{t_{1}}^{t_{2}}f(u) \Big\{\int_{\Sigma_{t}}\Big[\frac{f^{\prime}(u)}{f(u)}|\nabla u|H-\frac{1}{2}\Big(\frac{|\nabla^{2}u|^{2}}{|\nabla u|^{2}}+R_{g}\Big)\Big]+2\pi \chi(\Sigma_{t}) \Big\}.
\end{equation} 
Here the mean curvature $ H $ of $ \partial \Omega $ is taken with respect to the unit normal $ \zeta $ pointing out of $ \Omega $, the mean curvature $ H $ of a regular level set $ \Sigma_{t} $ is taken with respect to $ \nu=|\nabla u|^{-1}\nabla u $, $ \chi (\Sigma_{t}) $ is the Euler characteristic of $ \Sigma_{t} $, $ t_{1}=\min_{\Omega} u $, and $ t_{2}=\max_{\Omega} u $.
\end{lemma}
\begin{proof}
For any constant $ \epsilon>0 $, one has
$$
div\Big(f(u)\nabla \sqrt{|\nabla u|^{2}+\epsilon}\Big)=f(u)\Delta \sqrt{|\nabla u|^{2}+\epsilon}+f^{\prime}(u)\frac{\nabla^{2}u(\nabla u,\nabla u)}{\sqrt{|\nabla u|^{2}+\epsilon}}.
$$
Therefore,
\begin{equation}\label{f(u) integral identity}
\int_{\partial \Omega}f(u)\partial_{\zeta}\sqrt{|\nabla u|^{2}+\epsilon}=\int_{\Omega}f(u)\Delta \sqrt{|\nabla u|^{2}+\epsilon}+\int_{\Omega}f^{\prime}(u)\frac{\nabla^{2}u(\nabla u,\nabla u)}{\sqrt{|\nabla u|^{2}+\epsilon}}.
\end{equation}
Notice that the mean curvature of $ \Sigma_{t} $ is 
\begin{equation}\label{H}
H=-\frac{1}{|\nabla u|}\Big\langle \nabla |\nabla u|,\frac{\nabla u}{|\nabla u|} \Big\rangle=-\frac{1}{|\nabla u|^{3}}\nabla^{2}u(\nabla u,\nabla u).
\end{equation}
Since $ u $ is constant on each connected component of $ \partial \Omega $, the mean curvature of $ \partial \Omega $ is
$$
H=-\frac{1}{|\nabla u|}\Big\langle \nabla |\nabla u|,\frac{\nabla u}{|\nabla u|} \Big\rangle=-\frac{1}{|\nabla u|}\partial_{\zeta}|\nabla u|.
$$
Thus we have 
$$
\partial_{\zeta}\sqrt{|\nabla u|^{2}+\epsilon}=-\frac{|\nabla u|^{2}}{\sqrt{|\nabla u|^{2}+\epsilon}}H.
$$
Hence,
\begin{equation}\label{f H}
\lim\limits_{\epsilon\to 0}\int_{\partial \Omega}f(u)\partial_{\zeta}\sqrt{|\nabla u|^{2}+\epsilon}=-\int_{\partial \Omega}f(u)|\nabla u|H.
\end{equation}
Similar to the proof of Lemma~\ref{f(u)1}, taking $ \epsilon \to 0 $ in the third term in (\ref{f(u) integral identity}) and using (\ref{H}) gives
\begin{equation}\label{f hess}
\begin{aligned}
\lim\limits_{\epsilon\to 0}\int_{\Omega}f^{\prime}(u)\frac{\nabla^{2}u(\nabla u, \nabla u)}{\sqrt{|\nabla u|^{2}+\epsilon}}&=\int_{\{|\nabla u|\neq 0\}}f^{\prime}(u)\frac{\nabla^{2}u(\nabla u, \nabla u)}{|\nabla u|}\\
&=-\int_{\{|\nabla u|\neq 0\}}f^{\prime}(u)H|\nabla u|^{2}\\
&=-\int_{t_{1}}^{t_{2}}f^{\prime}(u)\int_{\Sigma_{t}}|\nabla u|H.
\end{aligned}
\end{equation}
To deal with the second term in (\ref{f(u) integral identity}), we follow the method of Stern \cite{Stern} (see also \cite{Miao}).

Let $ \mathcal{C} $ denote the set of critical values of $ u $ in $ [t_{1},t_{2}] $. Let $ W $ denote an open set of $ [t_{1},t_{2}] $ such that $ W $ contains $ \mathcal{C} $. Let $ D $ be the complement of $ W $ in  $ [t_{1},t_{2}] $.

On $ u^{-1}(D) $, $ f(u)\Delta \sqrt{|\nabla u|^{2}+\epsilon} $ is integrable. By coarea formula,
$$
\int_{u^{-1}(D)}f(u)\Delta \sqrt{|\nabla u|^{2}+\epsilon}=\int_{D}\int_{\Sigma_{t}}f(u)\frac{\Delta \sqrt{|\nabla u|^{2}+\epsilon}}{|\nabla u|}.
$$
By Stern's formula \cite{Stern},
$$
\Delta \sqrt{|\nabla u|^{2}+\epsilon}\geq \frac{1}{2\sqrt{|\nabla u|^{2}+\epsilon}}[|\nabla^{2}u|^{2}+(R_{g}-2K_{\Sigma_{t}})|\nabla u|^{2}],
$$
where $ K_{\Sigma_{t}} $ is the Gauss curvature of $ \Sigma_{t} $. Hence,
\begin{equation}\label{Delta f}
\int_{u^{-1}(D)}f(u)\Delta \sqrt{|\nabla u|^{2}+\epsilon}\geq \int_{D}\int_{\Sigma_{t}}\frac{f(u)}{|\nabla u|}\frac{[|\nabla^{2}u|^{2}+(R_{g}-2K_{\Sigma_{t}})|\nabla u|^{2}]}{2\sqrt{|\nabla u|^{2}+\epsilon}}.
\end{equation}
With $ W $ fixed, letting $ \epsilon\to 0 $ in (\ref{Delta f}) and using Gauss-Bonnet theorem, we get
\begin{equation}\label{D}
\begin{aligned}
\lim\limits\inf_{\epsilon\to 0}\int_{u^{-1}(D)}f(u)\Delta \sqrt{|\nabla u|^{2}+\epsilon}&\geq \int_{D}\int_{\Sigma_{t}}f(u)\frac{[|\nabla^{2}u|^{2}+(R_{g}-2K_{\Sigma_{t}})|\nabla u|^{2}]}{2|\nabla u|^{2}}\\
&=\int_{D}f(t)\Big[\int_{\Sigma_{t}}\frac{1}{2}\Big(\frac{|\nabla^{2}u|^{2}}{|\nabla u|^{2}}+R_{g}\Big)-2\pi \chi(\Sigma_{t}) \Big].
\end{aligned}
\end{equation}
To estimate the integral on $ u^{-1}(W) $, we note that 
$$
\begin{aligned}
\Delta \sqrt{|\nabla u|^{2}+\epsilon}&=div\Big[\frac{\nabla |\nabla u|^{2}}{2\sqrt{|\nabla u|^{2}+\epsilon}}\Big]\\
&=\frac{1}{\sqrt{|\nabla u|^{2}+\epsilon}}\Big[|\nabla^{2}u|^{2}+\mathrm{Ric}(\nabla u,\nabla u)-\frac{|\nabla u|^{2}}{|\nabla u|^{2}+\epsilon}|\nabla |\nabla u||^{2} \Big]\\
&\geq \frac{1}{\sqrt{|\nabla u|^{2}+\epsilon}}\mathrm{Ric}(\nabla u,\nabla u).
\end{aligned}
$$
This implies that 
\begin{equation}\label{W}
\begin{aligned}
\int_{u^{-1}(W)}f(u)\Delta \sqrt{|\nabla u|^{2}+\epsilon}&\geq -\max_{\Omega}|\mathrm{Ric}|\int_{u^{-1}(W)}\frac{|\nabla u|^{2}}{\sqrt{|\nabla u|^{2}+\epsilon}}f(u)\\
&\geq -\max_{\Omega}|\mathrm{Ric}|\int_{u^{-1}(W)}|\nabla u|f(u)\\
&=-\max_{\Omega}|\mathrm{Ric}|\int_{W}\int_{\Sigma_{t}}f(u).
\end{aligned}
\end{equation}
Combining (\ref{D}) with (\ref{W}), we obtain 
$$
\lim\limits_{\epsilon\to 0}\int_{\Omega}f(u)\Delta \sqrt{|\nabla u|^{2}+\epsilon}\geq \int_{D}f(t)\Big[\int_{\Sigma_{t}}\frac{1}{2}\Big(\frac{|\nabla^{2}u|^{2}}{|\nabla u|^{2}}+R_{g}\Big)-2\pi \chi(\Sigma_{t}) \Big]-\max_{\Omega}|\mathrm{Ric}|\int_{W}\int_{\Sigma_{t}}f(u).
$$
Since $ \int_{\Omega}|\nabla u|f(u)<\infty $, we can choose $ W $, such that the measure of $ W $ is arbitrarly small. Then one has
\begin{equation}\label{f Delta}
\lim\limits_{\epsilon\to 0}\int_{\Omega}f(u)\Delta \sqrt{|\nabla u|^{2}+\epsilon}\geq \int_{t_{1}}^{t_{2}}f(t)\Big[\int_{\Sigma_{t}}\frac{1}{2}\Big(\frac{|\nabla^{2}u|^{2}}{|\nabla u|^{2}}+R_{g}\Big)-2\pi \chi(\Sigma_{t}) \Big].
\end{equation}
From (\ref{f(u) integral identity}), (\ref{f H}), (\ref{f hess}) and (\ref{f Delta}), we obtain the desired result. 
\end{proof}

\section{\bf Monotone quantities}
In this section, our aim is to prove Theorem~\ref{th 1.1} and Theorem~\ref{th 1.2}. We derive monotone quantities for harmonic functions on asymptotically flat 3-manifolds with simple topology and non
negative scalar curvature. These monotone quantities are constant
on spatial Schwarzschild manifolds outside rotationally symmetric spheres. To derive monotone quantities, our method is different from the ODE analysis in \cite{Xia} and \cite{Maz}, we follow the strategy developed in the work of \cite{Miao}, we first establish the monotone quantity (\ref{D t}) then derive monotone quantities (\ref{P}) and recover two monotone quantities (\ref{Q}) and (\ref{F}) in \cite{Miao 0} and \cite{Xia}, respectively. Then, we apply the monotone quantity (\ref{F}) to derive the monotone quantity (\ref{G}). 

The following monotone quantitity $ \mathcal{D}_{k}(t) $ was investigated in \cite{Miao, Hrisch1} for p-harmonic functions with $ p\in (1,2] $ and $ k=1 $. In \cite{Xia}, Xia-Yin-Zhou studied $ \mathcal{D}_{k}(t) $ for p-harmonic functions with $ p\in (1,3) $ and $ k\in (0,2] $. For convenience, we reprove this result in the case of harmonic functions for $ k>0 $.

\begin{lemma}\label{D mono}
Let $ (\Omega,g) $ be a compact, Riemannian 3-manifold with boundary $ \partial \Omega $. Suppose $ \partial \Omega $ has two connected components $ S_{1} $ and $ S_{2} $. Let $ u $ be a harmonic function on $ (\Omega,g) $ such that $ u=c_{i} $ on $ S_{i} $, $ i=1,2 $, where $ c_{1}, c_{2} $ are constants with $ c_{1}<c_{2}<1 $. If $ R_{g}\geq 0 $ and the level sets $ \Sigma_{t}:=u^{-1}(t) $ is connected for $ t\in [c_{1},c_{2}] $, then
\begin{equation}\label{D t}
\mathcal{D}_{k}(t)=\frac{4\pi(1-t)}{k\omega(t)}+\frac{1}{\omega(t)^{2}}\int_{\Sigma_{t}}|\nabla u|H-\frac{4\omega(t)-k}{(1-t)\omega(t)^{3}}\int_{\Sigma_{t}}|\nabla u|^{2}
\end{equation}
is nonincreasing in $ t $. Here $ t\in [c_{1},c_{2}] $ denotes a regular value of $ u $ and $ H $ is the mean curvature of $ \Sigma_{t} $ with respect to the unit normal $ \nu=|\nabla u|^{-1}\nabla u $.
\end{lemma}
\begin{proof}
Notice that $ R_{g}\geq 0 $ and $ \Sigma_{t} $ are connected. By Lemma~\ref{deriva}, at a regular value $ t $,  we have 
\begin{equation}\label{de leq}
\Big(\int_{\Sigma_{t}}|\nabla u|H\Big)^{\prime}\leq 4\pi-\frac{3}{4}\int_{\Sigma_{t}}H^{2},
\end{equation}
where we have used the Gauss-Bonnet theorem. Hence, at a regular value $ t $, we compute
$$
\begin{aligned}
-\mathcal{D}_{k}^{\prime}(t)&=\frac{1}{\omega(t)^{2}}\Big[4\pi+\frac{3(2\omega(t)-k)^{2}}{(1-t)^{2}\omega(t)^{2}}\int_{\Sigma_{t}}|\nabla u|^{2}-\frac{6\omega(t)-3k}{(1-t)\omega(t)}\int_{\Sigma_{t}}|\nabla u|H-\Big(\int_{\Sigma_{t}}|\nabla u|H\Big)^{\prime}\Big]\\
&\geq \frac{1}{\omega(t)^{2}}\Big[\frac{3(2\omega(t)-k)^{2}}{(1-t)^{2}\omega(t)^{2}}\int_{\Sigma_{t}}|\nabla u|^{2}-\frac{6\omega(t)-3k}{(1-t)\omega(t)}\int_{\Sigma_{t}}|\nabla u|H+\frac{3}{4}\int_{\Sigma_{t}}H^{2}\Big]\\
&=\frac{1}{\omega(t)^{2}}\int_{\Sigma_{t}}\frac{3}{4}\Big(H-\frac{2(2\omega-k)}{(1-u)\omega}|\nabla u|\Big)^{2}\\
&\geq 0.
\end{aligned}
$$
In general, let $ t_{1}<t_{2} $ be two regular values of $ u $. On $ \Omega_{[t_{1},t_{2}]}=\{x\in \Omega\,|\, t_{1}\leq u\leq t_{2} \} $, taking $ f(u)=\frac{1}{\omega(u)^{2}} $ in Lemma~\ref{f(u)2}, we get
\begin{equation}\label{H f}
\begin{aligned}
&\frac{1}{\omega(t_{1})^{2}}\int_{\Sigma_{t_{1}}}|\nabla u|H-\frac{1}{\omega(t_{2})^{2}}\int_{\Sigma_{t_{2}}}|\nabla u|H\\
&\geq \int_{t_{1}}^{t_{2}}\frac{1}{\omega(t)^{2}}\Big[\int_{\Sigma_{t}}\frac{1}{2}\Big(\frac{|\nabla^{2}u|^{2}}{|\nabla u|^{2}}+R_{g}+\frac{4(k-1)}{\omega(u)}|\nabla u|H\Big)-2\pi\chi(\Sigma_{t})\Big].
\end{aligned}
\end{equation}
Let $ \mathbb{II} $ denote the second fundamental form of $ \Sigma_{t} $ with respect to $ \nu $. Let $ X, Y $ be the vectors tangent to $ \Sigma_{t} $. Along $ \Sigma_{t} $, we have
$$
\nabla^{2}u(X,Y)=|\nabla u|\mathbb{II}(X,Y),\,\,\,\nabla^{2}u(X,\nu)=X(|\nabla u|),\,\,\,\nabla^{2}u(\nu,\nu)=-H|\nabla u|.
$$
Hence, we get
\begin{equation}\label{hess}
\begin{aligned}
\frac{|\nabla^{2}u|^{2}}{|\nabla u|^{2}}&=|\mathbb{II}|^{2}+2\frac{|\nabla^{\Sigma_{t}}|\nabla u||^{2}}{|\nabla u|^{2}}+H^{2}\\
&=|\overset{\circ}{\mathbb{II}}|^{2}+2\frac{|\nabla^{\Sigma_{t}}|\nabla u||^{2}}{|\nabla u|^{2}}+\frac{3}{2}H^{2},
\end{aligned}
\end{equation}
where $ \overset{\circ}{\mathbb{II}} $ denotes the traceless part of $ \mathbb{II} $.
By (\ref{H f}) and (\ref{hess}), one has
\begin{equation}\label{H f1}
\begin{aligned}
&\frac{1}{\omega(t_{1})^{2}}\int_{\Sigma_{t_{1}}}|\nabla u|H-\frac{1}{\omega(t_{2})^{2}}\int_{\Sigma_{t_{2}}}|\nabla u|H\\
&\geq \int_{t_{1}}^{t_{2}}\frac{1}{\omega(t)^{2}}\Big[\int_{\Sigma_{t}}\Big(\frac{1}{2}|\overset{\circ}{\mathbb{II}}|^{2}+\frac{|\nabla^{\Sigma_{t}}|\nabla u||^{2}}{|\nabla u|^{2}}+\frac{3}{4}H^{2}+\frac{1}{2}R_{g}+\frac{2(k-1)}{\omega(u)}|\nabla u|H\Big)-2\pi\chi(\Sigma_{t}) \Big].
\end{aligned}
\end{equation}
Now, we replace term $ H^{2} $ in (\ref{H f1}) with 
$$
\Big(H-\frac{2(2\omega-k)}{(1-u)\omega}|\nabla u|\Big)^{2}. 
$$
Thus, we can rewrite (\ref{H f1}) as 
 \begin{equation}
 \begin{aligned}
 &\frac{1}{\omega(t_{1})^{2}}\int_{\Sigma_{t_{1}}}|\nabla u|H-\frac{1}{\omega(t_{2})^{2}}\int_{\Sigma_{t_{2}}}|\nabla u|H\\
 &\geq \int_{t_{1}}^{t_{2}}\frac{1}{\omega(t)^{2}}\Big[\int_{\Sigma_{t}}\Big(\frac{1}{2}|\overset{\circ}{\mathbb{II}}|^{2}+\frac{|\nabla^{\Sigma_{t}}|\nabla u||^{2}}{|\nabla u|^{2}}+\frac{3}{4}\Big(H-\frac{2(2\omega-k)}{(1-u)\omega}|\nabla u|\Big)^{2}\\
 &\,\,\,\,\,\,+\frac{4\omega-k}{(1-u)\omega}|\nabla u|H-\frac{3(2\omega-k)^{2}}{(1-u)^{2}\omega^{2}}
 +\frac{1}{2}R_{g}\Big)-2\pi\chi(\Sigma_{t}) \Big].
 \end{aligned}
\end{equation}
Under the assumption that $ \Sigma_{t} $ is connected, we get
\begin{equation}\label{H f2}
\begin{aligned}
&\frac{4\pi}{k}\Big(\frac{1-t_{1}}{\omega(t_{1})}-\frac{1-t_{2}}{\omega(t_{2})}\Big)+\frac{1}{\omega(t_{1})^{2}}\int_{\Sigma_{t_{1}}}|\nabla u|H-\frac{1}{\omega(t_{2})^{2}}\int_{\Sigma_{t_{2}}}|\nabla u|H\\
&\geq \int_{t_{1}}^{t_{2}}\frac{1}{\omega(t)^{2}}\int_{\Sigma_{t}}\Big(\frac{1}{2}|\overset{\circ}{\mathbb{II}}|^{2}+\frac{|\nabla^{\Sigma_{t}}|\nabla u||^{2}}{|\nabla u|^{2}}+\frac{3}{4}\Big(H-\frac{2(2\omega-k)}{(1-u)\omega}|\nabla u|\Big)^{2}\\
&\,\,\,\,\,\,+\frac{4\omega-k}{(1-u)\omega}|\nabla u|H-\frac{3(2\omega-k)^{2}}{(1-u)^{2}\omega^{2}}
+\frac{1}{2}R_{g}\Big).
\end{aligned}
\end{equation}
On the other hand, taking $ f(u)=\frac{4\omega(u)-k}{(1-t)\omega(u)^{3}} $ in Lemma~\ref{f(u)1} gives
\begin{equation}\label{f 3}
\begin{aligned}
&\frac{4\omega(t_{2})-k}{(1-t_{2})\omega(t_{2})^{3}}\int_{\Sigma_{t_{2}}}|\nabla u|^{2}-\frac{4\omega(t_{1})-k}{(1-t_{1})\omega(t_{1})^{3}}\int_{\Sigma_{t_{1}}}|\nabla u|^{2}\\
&=\int_{\Omega}\frac{3(2\omega-k)^{2}}{(1-u)^{2}\omega^{4}}|\nabla u|^{3}+\int_{\{|\nabla u|\neq 0\}} \frac{4\omega-k}{(1-u)\omega^{3}}\frac{\nabla^{2}u(\nabla u, \nabla u)}{|\nabla u|}\\
&=\int_{\Omega}\frac{3(2\omega-k)^{2}}{(1-u)^{2}\omega^{4}}|\nabla u|^{3}-\int_{\{|\nabla u|\neq 0\}} \frac{4\omega-k}{(1-u)\omega^{3}}|\nabla u|^{2}H\\
&=\int_{t_{1}}^{t_{2}}\frac{1}{\omega(t)^{2}}\int_{\Sigma_{t}}\Big[\frac{3(2\omega-k)^{2}}{(1-u)^{2}\omega^{2}}|\nabla u|^{2}-\frac{4\omega-k}{(1-u)\omega}|\nabla u|H\Big].
\end{aligned}
\end{equation}
It follows from (\ref{H f2}) and (\ref{f 3}) that 
\begin{equation}\label{mono Psi}
\mathcal{D}_{k}(t_{1})-\mathcal{D}_{k}(t_{2})\geq \int_{t_{1}}^{t_{2}}\frac{1}{\omega(t)^{2}}\int_{\Sigma_{t}}\Big(\frac{1}{2}|\overset{\circ}{\mathbb{II}}|^{2}+\frac{|\nabla^{\Sigma_{t}}|\nabla u||^{2}}{|\nabla u|^{2}}+\frac{3}{4}\Big(H-\frac{2(2\omega-k)}{(1-u)\omega}|\nabla u|\Big)^{2}+\frac{1}{2}R_{g}\Big).
\end{equation}
Since $ R_{g}\geq 0 $, we conclude that $ \mathcal{D}_{k}(t_{1})\geq \mathcal{D}_{k}(t_{2}) $. This completes the proof of Lemma~\ref{Psi}.
\end{proof}
Applying the monotonicity of $ \mathcal{D}_{k}(t) $, one can recover the following geometric inequalities (\ref{t cp}) in \cite{Xia} and (\ref{0 cp}) in \cite{Miao}, respectively.

\begin{proposition}\label{pro}
Let $ (M,g) $ be a complete, orientable, one-ended asymptotically flat 3-manifold with compact, connnected boundary $ \Sigma $. Suppose that $ H_{2}(M,\Sigma)=0 $. Let $ u $ be the solution of (\ref{u}). If $ R_{g}\geq 0 $, then
\begin{equation}\label{t cp}
\frac{4\pi}{k}+\frac{1}{(1-t)\omega(t)}\int_{\Sigma_{t}}|\nabla u|H\geq \frac{4\omega(t)-k}{(1-t)^{2}\omega(t)^{2}}\int_{\Sigma_{t}}|\nabla u|^{2}.
\end{equation}
In particular, at $ \Sigma $,
\begin{equation}\label{0 cp}
4\pi+k\int_{\Sigma}|\nabla u|H\geq k(4-k)\int_{\Sigma}|\nabla u|^{2},
\end{equation}
and equality holds if and only if $ (M,g) $ is isometric to $ (M_{m,r_{0}},g_{m}) $ of mass $ \mathfrak{m}=2r_{0}(k-1) $ with $ r_{0}=\frac{\mathfrak{c}_{\Sigma}}{k} $.
\end{proposition}

\begin{proof}
Let $ \mathcal{D}_{k}(t) $ be defined as in (\ref{Psi}). Since $ (M,g) $ is asymptotically flat,  By Lemma~\ref{funda u} $ (iv) $, we have
$$
\lim\limits_{t\to1}\mathcal{D}_{k}(t)=0.
$$
Taking $ t_{2}\to 1 $ in (\ref{mono Psi}), we get
\begin{equation}\label{Psi geq0}
\mathcal{D}_{k}(t)\geq \int_{t}^{1}\frac{1}{\omega(s)^{2}}\int_{\Sigma_{s}}\Big(\frac{1}{2}|\overset{\circ}{\mathbb{II}}|^{2}+\frac{|\nabla^{\Sigma_{s}}|\nabla u||^{2}}{|\nabla u|^{2}}+\frac{3}{4}\Big(H-\frac{2(2\omega-k)}{(1-u)\omega}|\nabla u|\Big)^{2}+\frac{1}{2}R_{g}\Big)\geq 0.
\end{equation}
Hence, (\ref{t cp}) follows from (\ref{Psi geq0}). In fact, since $ k>0 $, we have $ \omega(t)>0 $ for $ t\in [0, 1) $. Therefore,  
$$
\frac{\omega(t)}{1-t}\mathcal{D}(t)\geq0.
$$
To prove the rigidity case of (\ref{0 cp}), suppose that the equality in (\ref{0 cp}) holds. From (\ref{Psi geq0}) and its proof, we conclude that for every regular value $ t\in [0,1] $, $ \Sigma_{t} $ is connected with $ \chi(\Sigma_{t})=2 $, thus $ \Sigma_{t} $ is a 2-sphere. Moreover, $ R_{g}=0 $, $ |\nabla^{\Sigma_{t}}|\nabla u||\equiv0 $, thus $ |\nabla u| $ only depends on $ t $, and 
\begin{equation}\label{key}
H=\frac{2(2\omega-k)}{(1-u)\omega}|\nabla u|=\frac{2(2\omega(t)-k)}{(1-t)\omega(t)}|\nabla u|.
\end{equation}
To show $ (M,g) $ is isometric to a exterior spatial Schwarzschild manifold of mass $ \mathfrak{m}=2(1-\frac{1}{k})\mathfrak{c} $, we start from a neighborhood of the boundary $ \Sigma $. Locally, $ g $ takes the form of 
$$
g=\eta(t)^{-2}dt^{2}+\gamma_{t}
$$ 
near $ \Sigma $, where $ t=u $, $ \eta(t)=|\nabla u| $ and $ \gamma_{t} $ denotes the induced metric on $ \Sigma_{t} $. Since 
$$ 
\overset{\circ}{\mathbb{II}}=0,\,\,\,\,\,   H=\frac{2(2\omega(t)-k)}{(1-t)\omega(t)}\eta(t), 
$$ 
we get 
$$
\partial_{t}\gamma_{t}=2\eta(t)^{-1}\mathbb{II}_{t}=\eta(t)^{-1}H\gamma_{t}=\frac{2(2\omega(t)-k)}{(1-t)\omega(t)}\gamma_{t}.
$$
Thus, $ (1-t)^{2}\omega(t)^{2}\gamma_{t}= $ a fixed metric. Since $ |\nabla u|^{\prime}=-H $ and $ \eta(t) $ satisfies 
$$
\eta^{\prime}(t)=-\frac{2(2\omega(t)-k)}{(1-t)\omega(t)}\eta(t),
$$
one has 
$$
|\nabla u|=\eta(t)=a(1-u)^{2}\omega(u)^{2}.
$$
Notice that
$$
|\nabla u|=c_{\Sigma}|x|^{-2}+O(|x|^{-2-\tau}),
$$
it follows
$$
\mathfrak{c}_{\Sigma}=\lim\limits_{|x|\to \infty}|x|^{2}|\nabla u|=a\lim\limits_{|x|\to \infty}|x|^{2}(1-u)^{2}\omega(u)^{2}=ak^{2}\mathfrak{c}_{\Sigma}^{2}.
$$
Thus $ a=(k^{2}\mathfrak{c}_{\Sigma})^{-1} $ and we have
\begin{equation}\label{grad u}
|\nabla u|=\frac{(1-u)^{2}\omega(u)^{2}}{k^{2}\mathfrak{c}_{\Sigma}},\,\,\,\,\,H=\frac{2(2\omega(u)-k)(1-u)\omega(u)}{k^{2}\mathfrak{c}_{\Sigma}}.
\end{equation}
From (\ref{grad u}), we see that $ |\nabla u|\neq 0 $  everywhere. Thus all the level sets of $ u $ are regular and diffeomorphic to each other. Now, up to an isometry, we conclude that $ M=\Sigma \times [0,1)  $, every slice $ {t}\times \Sigma $ is the level set $ \{u=t\} $ and the metric $ g $ can be written as 
$$
g=\frac{(k^{2}\mathfrak{c}_{\Sigma})^{2}}{(1-t)^{4}\omega(t)^{4}}dt^{2}+\frac{1}{(1-t)^{2}\omega(t)^{2}}\sigma_{0}
$$
for some fixed metric $ \sigma_{0} $ on the 2-sphere $ \Sigma $. We claim that 
$$
\sigma_{0}=k^{2}\mathfrak{c}_{\Sigma}^{2}g_{S^{2}},
$$
where $ g_{S^{2}} $ is the standard metric on the unit 2-sphere. In fact, by Stern's Bochnner formula \cite{Stern}, we have 
$$
\frac{\Delta |\nabla u|}{|\nabla u|}=\frac{1}{2}\Big(\frac{|\nabla^{2}u|^{2}}{|\nabla u|^{2}}+R_{g}-R_{\Sigma_{t}}\Big).
$$
Notice that $ R_{g}\equiv 0 $ and 
$$
\frac{|\nabla^{2}u|^{2}}{|\nabla u|^{2}}=|\mathbb{II}|^{2}+2\frac{|\nabla^{\Sigma_{t}}|\nabla u||^{2}}{|\nabla u|^{2}}+H^{2}=\frac{3}{2}H^{2},
$$
since $ |\nabla^{\Sigma_{t}}|\nabla u||\equiv0 $ and $ \overset{\circ}{\mathbb{II}}=0 $. On the other hand, by (\ref{grad u}), we have 
$$
\nabla |\nabla u|=-\frac{2(2\omega-k)(1-u)\omega}{k^{2}\mathfrak{c}_{\Sigma}}\nabla u,\,\,\,\,\,\frac{\Delta |\nabla u|}{|\nabla u|}=\frac{2}{k^{4}\mathfrak{c}_{\Sigma}^{2}}(6\omega^{2}-6k\omega+k^{2})(1-u)^{2}\omega^{2}.
$$
Therefore, we deduce that 
$$
R_{\Sigma_{t}}=\frac{2(1-u)^{2}\omega^{2}}{k^{2}\mathfrak{c}_{\Sigma}^{2}}.
$$
In particular, 
$$ 
R_{\Sigma}=\frac{2}{k^{2}\mathfrak{c}_{\Sigma}^{2}}. 
$$
Since $ \Sigma $ is diffeomorphic to a 2-sphere, we conclude that $
\sigma_{0}=k^{2}\mathfrak{c}_{\Sigma}^{2}g_{S^{2}} $. Consequently, we conclude that $ (M,g) $ is isometric to   
$$
\Big(\Sigma \times [0,1), g=\frac{(k^{2}\mathfrak{c}_{\Sigma})^{2}}{(1-t)^{4}\omega(t)^{4}}dt^{2}+\frac{k^{2}\mathfrak{c}_{\Sigma}^{2}}{(1-t)^{2}\omega(t)^{2}}g_{S^{2}} \Big).
$$
Under a change of variable 
$$
r=\frac{\mathfrak{c}_{\Sigma}}{k}\frac{\omega(t)}{1-t}\in [\frac{\mathfrak{c}_{\Sigma}}{k}, +\infty), 
$$
one gets that $ (M,g) $ is isometric to 
$$
\Big(\mathbb{R}^{3}\setminus B_{r_{0}}(0),(1+\frac{m_{0}}{2|x|})^{4}g_{\mathbb{R}^{3}} \Big),
$$
where  $ r_{0}=\frac{\mathfrak{c}_{\Sigma}}{k} $, $ m_{0}=2r_{0}(k-1)=\frac{2(k-1)\mathfrak{c}_{\Sigma}}{k} $.
\end{proof}
To prove Theorem~\ref{th 1.2}, we need the following Hawking mass estimate.
\begin{lemma}\label{mass estimate}
	Let $ u $ be the solution of (\ref{u}), we have the following
\begin{equation}\label{lim}
	\lim\limits_{t\to 1}\frac{1}{1-t}\Big(1-\frac{1}{16\pi}\int_{\Sigma_{t}}H^{2}\Big)\leq 2\mathfrak{m}\mathfrak{c}_{\Sigma}^{-1}.
\end{equation}
\end{lemma}
\begin{proof}
By the work in \cite{Ag2}, the Hawking mass estimate 
$$
\lim\limits_{t\to \infty}\frac{t}{4}\Big(16\pi-\int\limits_{\{u=\alpha_{p}(t)\}}H^{2}\Big)\leq 8\pi \mathfrak{m}
$$
was proved, where $ \alpha_{p}(t)=1-(\frac{t_{p}}{t})^{\frac{3-p}{p-1}} $, $ t_{p}=\Big(\frac{p-1}{3-p}c_{p}\Big)^{\frac{p-1}{3-p}} $, $ c_{p} $ is the $ p $ capacity of $ \Sigma $. In our setting, take $ p=2 $, then $ \alpha_{2}(t)=1-\frac{\mathfrak{c}_{\Sigma}}{t} $. Now, consider a change of variable $ s=1-\frac{\mathfrak{c}_{\Sigma}}{t} $, which implies (\ref{lim}). 
\end{proof}
Combining Proposition~\ref{pro} with Lemma~\ref{mass estimate}, we can prove the monotonicity and geometric inequalities of $ \mathcal{A}_{k}(t) $ and $ \mathcal{B}_{k}(t) $.

\textbf{Proof of Theorem 1.1}.
First we prove $ (i) $. At a regular value $ t $, we compute
\begin{equation}\label{A pri}
\begin{aligned}
\mathcal{A}_{k}^{\prime}(t)=&\frac{k}{(1-t)^{2}(4\omega(t)-k)}\Big[g(t)\frac{4\omega(t)-k}{k}\pi+g^{\prime}(t)\frac{(1-t)(4\omega(t)-k)}{k}\pi\\
&-\frac{6(2\omega(t)-k)}{(1-t)(4\omega(t)-k)}\int_{\Sigma_{t}}H|\nabla u|-\Big(\int_{\Sigma_{t}}H|\nabla u|\Big)^{\prime} \Big].
\end{aligned}
\end{equation}
By (\ref{de leq}), we get
\begin{equation}
\begin{aligned}
\mathcal{A}_{k}^{\prime}(t)\geq& \frac{k}{(1-t)^{2}(4\omega(t)-k)}\Big[g(t)\frac{4\omega(t)-k}{k}\pi+g^{\prime}(t)\frac{(1-t)(4\omega(t)-k)}{k}\pi\\
&-\frac{6t}{(1-t)(2t+1)}\int_{\Sigma_{t}}H|\nabla u|-4\pi+\frac{3}{4}\int_{\Sigma_{t}}H^{2} \Big]\\
=&\frac{k}{(1-t)^{2}(4\omega(t)-k)}\Big[g(t)\frac{4\omega(t)-k}{k}\pi+g^{\prime}(t)\frac{(1-t)(4\omega(t)-k)}{k}\pi-4\pi\\
&+\frac{3}{4}\int_{\Sigma_{t}}\Big(H-\frac{2(2\omega-k)}{(1-u)\omega}|\nabla u|\Big)^{2}+\frac{3(2\omega(t)-k)^{2}}{(1-t)\omega(t)(4\omega(t)-k)}\int_{\Sigma_{t}}H|\nabla u|\\
&-\frac{3(2\omega(t)-k)^{2}}{(1-t)^{2}\omega(t)^{2}}\int_{\Sigma_{t}}|\nabla u|^{2}\Big].\\
\end{aligned}
\end{equation}
Since $ 0<k<4 $, we have $ \omega(t)>0 $ and $ 4\omega(t)-k>0 $ for $ t\in [0,1) $. From (\ref{t cp}), we deduce that 
\begin{equation}\label{t cp1}
\frac{3(2\omega(t)-k)^{2}}{(1-t)\omega(t)(4\omega(t)-k)}\int_{\Sigma_{t}}H|\nabla u|
-\frac{3(2\omega(t)-k)^{2}}{(1-t)^{2}\omega(t)^{2}}\int_{\Sigma_{t}}|\nabla u|^{2}+\frac{12\pi(2\omega(t)-k)^{2}}{k(4\omega(t)-k)}\geq 0.
\end{equation}
Moreover, by the definition of $ g(t) $, we get 
$$
g(t)=\frac{4\omega(t)-k}{k}-\frac{k}{4\omega(t)-k},\,\,\,\,\,g^{\prime}(t)=\frac{4(k-1)}{k}+\frac{4k(k-1)}{(4\omega(t)-k)^{2}}.
$$
This gives that
\begin{equation}\label{g t}
g(t)+g^{\prime}(t)(1-t)=3-\frac{8k\omega(t)-5k^{2}}{(4\omega(t)-k)^{2}}.
\end{equation}
Hence, it follows from (\ref{A pri}), (\ref{t cp1}) and (\ref{g t}) that 
$$
\begin{aligned}
\mathcal{A}_{k}^{\prime}(t)&\geq\frac{\pi}{(1-t)^{2}}\Big[g(t)+g^{\prime}(t)(1-t)-\frac{4}{4\omega(t)-k}\Big(k+\frac{3(2\omega(t)-k)^{2}}{4\omega(t)-k}\Big) \Big]\\
&=\frac{\pi}{(1-t)^{2}}\Big[g(t)+g^{\prime}(t)(1-t)-3+\frac{8k\omega(t)-5k^{2}}{(4\omega(t)-k)^{2}}\Big]\\
&=0.
\end{aligned}
$$
If $ u $ has critical values, we may apply a regularization argument to show that
$$
\mathcal{A}_{k}(t_{2})-\mathcal{A}_{k}(t_{1})\geq0
$$
for $ t_{2}>t_{1} $. In fact, we can take $$ f(u)=\frac{k}{(1-u)^{2}(4\omega(u)-k)} $$ in Lemma~\ref{f(u)2} and use (\ref{hess}) to find that
$$
\begin{aligned}
&-\frac{k}{(1-t_{2})^{2}(4\omega(t_{2})-k)}\int_{\Sigma_{t_{2}}}H|\nabla u|+\frac{k}{(1-t_{1})^{2}(4\omega(t_{1})-k)}\int_{\Sigma_{t_{1}}}H|\nabla u|\\
&\geq\int_{t_{1}}^{t_{2}}\frac{k}{(1-t)^{2}(4\omega(t)-k)}\Big[-\frac{6(2\omega(t)-k)}{(1-t)(4\omega(t)-k)}\int_{\Sigma_{t}}|\nabla u|H+\frac{1}{2}\int_{\Sigma_{t}}\Big(\frac{|\nabla^{2}u|^{2}}{|\nabla u|^{2}}+R_{g}\Big)-2\pi\chi(\Sigma_{t}) \Big]dt\\
&=\int_{t_{1}}^{t_{2}}\frac{k}{(1-t)^{2}(4\omega(t)-k)}\Big[-\frac{6(2\omega(t)-k)}{(1-t)(4\omega(t)-k)}\int_{\Sigma_{t}}|\nabla u|H\\
&\,\,\,\,\,\,\,\,\,\,\,\,\,\,\,\,\,\,\,\,\,\,\,\,\,\,\,\,\,\,\,\,\,\,\,\,\,\,\,+\frac{1}{2}\int_{\Sigma_{t}}\Big(|\overset{\circ}{\mathbb{II}}|^{2}+2\frac{|\nabla^{\Sigma_{t}}|\nabla u||^{2}}{|\nabla u|^{2}}+\frac{3}{2}H^{2}+R_{g}\Big)-2\pi\chi(\Sigma_{t}) \Big]dt\\
&\geq \int_{t_{1}}^{t_{2}}\frac{k}{(1-t)^{2}(4\omega(t)-k)}\Big[-\frac{6(2\omega(t)-k)}{(1-t)(4\omega(t)-k)}\int_{\Sigma_{t}}|\nabla u|H+\frac{3}{4}\int_{\Sigma_{t}}H^{2}-4\pi\Big]dt.
\end{aligned}
$$
Hence, we have
$$
\begin{aligned}
\mathcal{A}_{k}(t_{2})-\mathcal{A}_{k}(t_{1})\geq& \int_{t_{1}}^{t_{2}}\frac{k}{(1-t)^{2}(4\omega(t)-k)}\Big[g(t)\frac{4\omega(t)-k}{k}\pi+g^{\prime}(t)\frac{(1-t)(4\omega(t)-k)}{k}\pi\\
&-\frac{6(2\omega(t)-k)}{(1-t)(4\omega(t)-k)}\int_{\Sigma_{t}}|\nabla u|H-4\pi+\frac{3}{4}\int_{\Sigma_{t}}H^{2} \Big]dt\\
=&\frac{k}{(1-t)^{2}(4\omega(t)-k)}\Big[g(t)\frac{4\omega(t)-k}{k}\pi+g^{\prime}(t)\frac{(1-t)(4\omega(t)-k)}{k}\pi-4\pi\\
&+\frac{3}{4}\int_{\Sigma_{t}}\Big(H-\frac{2(2\omega-k)}{(1-u)\omega}|\nabla u|\Big)^{2}+\frac{3(2\omega(t)-k)^{2}}{(1-t)\omega(t)(4\omega(t)-k)}\int_{\Sigma_{t}}H|\nabla u|\\
&-\frac{3(2\omega(t)-k)^{2}}{(1-t)^{2}\omega(t)^{2}}\int_{\Sigma_{t}}|\nabla u|^{2}\Big]\\
\geq&0.
\end{aligned}
$$

To prove (\ref{lim A}), by (\ref{lim}),
$$
\lim\limits_{t\to 1}\frac{1}{1-t}\Big(1-\frac{1}{16\pi}\int_{\Sigma_{t}}H^{2}\Big)\leq 2\mathfrak{m}\mathfrak{c}_{\Sigma}^{-1}.
$$
Observe that 
\begin{equation}\label{will 1}
\begin{aligned}
\Big(\int_{\Sigma_{t}}H|\nabla u| \Big)^{\prime}&\leq 4\pi-\frac{3}{4}\int_{\Sigma_{t}}H^{2}\\
&=-8\pi+12\pi\Big(1-\frac{1}{16\pi}\int_{\Sigma_{t}}H^{2}\Big)\\
&\leq 24\pi \mathfrak{m}\mathfrak{c}_{\Sigma}^{-1}(1-t)-8\pi.
\end{aligned}
\end{equation}
Since $ \lim\limits_{t\to 1}\int_{\Sigma_{t}}H|\nabla u|=0 $, integrating (\ref{will 1}) from $ t $ to $ 1 $ gives
\begin{equation}\label{will 2}
-\int_{\Sigma_{t}}H|\nabla u|\leq -8\pi(1-t)+12\pi \mathfrak{m}\mathfrak{c}_{\Sigma}^{-1}(1-t)^{2}.
\end{equation}
This implies that 
$$
-\frac{k}{(1-t)(4\omega(t)-k)}\int_{\Sigma_{t}}H|\nabla u|\leq -\frac{8 \pi k}{4\omega(t)-k}+\frac{12\pi k \mathfrak{m}\mathfrak{c}_{\Sigma}^{-1}(1-t)}{4\omega(t)-k}.
$$
On the other hand, 
$$
\lim\limits_{t\to 1}\Big[g(t)\pi-\frac{8 \pi k}{4\omega(t)-k}+\frac{12\pi k \mathfrak{m}\mathfrak{c}_{\Sigma}^{-1}(1-t)}{4\omega(t)-k} \Big]=0.
$$
By L'Hospital rule, 
$$
\begin{aligned}
\lim\limits_{t\to 1}\mathcal{A}_{k}(t)&\leq \lim\limits_{t\to 1}\frac{1}{1-t}\Big[g(t)\pi-\frac{8 \pi k}{4\omega(t)-k}+\frac{12\pi k \mathfrak{m}\mathfrak{c}_{\Sigma}^{-1}(1-t)}{4\omega(t)-k} \Big]\\
&=\lim\limits_{t\to 1}\frac{\pi}{-1}\Big[\frac{4(k-1)\pi}{k}+\frac{4k(k-1)\pi}{(4\omega(t)-k)^{2}}+\frac{32\pi k(k-1)}{(4\omega(t)-k)^{2}}-\frac{12\pi k\mathfrak{m}\mathfrak{c}_{\Sigma}^{-1} }{4\omega(t)-k}\Big]\\
&=4\pi(\mathfrak{m}\mathfrak{c}_{\Sigma}^{-1}-\frac{2(k-1)}{k}).
\end{aligned}
$$
To prove $(ii)$,  we calculate, at a regular value $ t $,
$$
\begin{aligned}
\mathcal{B}_{k}^{\prime}(t)&=\frac{k^{2}}{(1-t)^{2}}\Big[\frac{4\pi}{k}-\frac{4\omega(t)-k}{(1-t)^{2}\omega(t)^{2}}\int_{\Sigma_{t}}|\nabla u|^{2}+\frac{1}{(1-t)\omega(t)}\int_{\Sigma_{t}}|\nabla u|H \Big]\geq0.
\end{aligned}
$$
where we have used (\ref{t cp}).
For the general case, we can again apply a regularization argument to show that 
$$
\mathcal{B}_{k}(t_{2})-\mathcal{B}_{k}(t_{1})\geq0
$$
for $ t_{2}>t_{1} $. Indeed, we can take $ f(u)=-\frac{k^{2}}{(1-u)^{3}\omega(u)} $ in Lemma~\ref{f(u)1} to obtain that
$$
\begin{aligned}
\mathcal{B}_{k}(t_{2})-\mathcal{B}_{k}(t_{1})&=\int_{t_{1}}^{t_{2}}\frac{4\pi k}{(1-t)^{2}}dt-\frac{k^{2}}{(1-t_{2})^{3}\omega(t_{2})}\int_{\Sigma_{t_{2}}}|\nabla u|^{2}+\frac{1}{(1-t_{1})^{3}\omega(t_{1})}\int_{\Sigma_{t_{1}}}|\nabla u|^{2}\\
&=\int_{t_{1}}^{t_{2}}\Big[\frac{4\pi k}{(1-t)^{2}}-\frac{k^{2}(4\omega(t)-k)}{(1-t)^{4}\omega(t)^{2}}\int_{\Sigma_{t}}|\nabla u|^{2}+\frac{k^{2}}{(1-t)^{3}\omega(t)}\int_{\Sigma_{t}}|\nabla u|H\Big]dt\\
&\geq0.
\end{aligned}
$$
To prove (\ref{lim B}), use (\ref{will 2}) to obtain
\begin{equation}\label{will 3}
\Big(\int_{\Sigma_{t}}|\nabla u|^{2}\Big)^{\prime}=-\int_{\Sigma_{t}}H|\nabla u|\leq -8\pi(1-t)+12\pi \mathfrak{m}\mathfrak{c}_{\Sigma}^{-1}(1-t)^{2}.
\end{equation}
Integrating (\ref{will 3}) from $ t $ to $ 1 $ gives
$$
-\int_{\Sigma_{t}}|\nabla u|^{2}\leq -4\pi(1-t)^{2}+4\pi \mathfrak{m}\mathfrak{c}_{\Sigma}^{-1}(1-t)^{3}.
$$
This implies that 
$$
-\frac{k^{2}}{(1-t)^{2}\omega(t)^{2}}\int_{\Sigma_{t}}|\nabla u|^{2}\leq -\frac{4\pi k^{2}}{\omega(t)^{2}}+\frac{4\pi k^{2} \mathfrak{m}\mathfrak{c}_{\Sigma}^{-1}(1-t)}{\omega(t)^{2}}.
$$
Using L'Hospital rule again, we have
$$
\begin{aligned}
\lim\limits_{t\to 1}\mathcal{B}_{k}(t)&\leq \lim\limits_{t\to 1}\frac{\omega(t)}{1-t}\Big[4\pi-\frac{4\pi k^{2}}{\omega(t)^{2}}+\frac{4\pi k^{2} \mathfrak{m}\mathfrak{c}_{\Sigma}^{-1}(1-t)}{\omega(t)^{2}} \Big]\\
&=\lim\limits_{t\to 1}\frac{k\pi}{-1}\Big[\frac{8\pi k^{2}(k-1)}{\omega(t)^{3}}-\frac{4\pi k^{2}\mathfrak{m}\mathfrak{c}_{\Sigma}^{-1}}{\omega(t)^{2}} \Big]\\
&=4\pi k(\mathfrak{m}\mathfrak{c}_{\Sigma}^{-1}-\frac{2(k-1)}{k}).
\end{aligned}
$$
$\hfill\square$
\begin{remark}
Comparing (\ref{0 cp}), (\ref{lim A0}) and (\ref{lim B0}), one has (\ref{0 cp})$+$(\ref{lim B0}) $\Rightarrow$ (\ref{lim A0}).
\end{remark}
The following monotone quantitity $ \mathcal{F}_{k}(t) $ was investigated in \cite{Ag2,Ag3} for p-harmonic function with $ p\in (1,3) $ and $ k=1 $. In \cite{Oronzio}, Oronzio considered $ \mathcal{F}_{k}(t) $ for harmonic functions with $ k=2 $. In \cite{Xia}, Xia-Yin-Zhou derived $ \mathcal{F}_{k}(t) $ for p-harmonic functions with $ p\in (1,3) $ and $ k\in (0,2] $. Here, we give a different proof of the monotonicity of $ \mathcal{F}_{k}(t) $ for harmonic functions with $ k>0 $ via using the monotone quantity $ \mathcal{D}_{k}(t) $ in Proposition~\ref{D mono}. In addition, the estimate of $ \lim\limits_{t\to 1}\mathcal{F}_{k}(t) $ is inspired by the idea in \cite{Ag2}, which is  slightly different from the methods in \cite{Oronzio} and \cite{Xia}.
\begin{proposition}\label{F mono}
Let $ (M,g) $ be a complete, orientable, one-ended asymptotically flat 3-manifold with compact, connnected boundary $ \Sigma $. Suppose that $ H_{2}(M,\Sigma)=0 $. Let $ u $ be the solution of (\ref{u}). If $ R_{g}\geq 0 $ and $ k>0 $, then the following quantity 
\begin{equation}\label{F}
\mathcal{F}_{k}(t)=\frac{4\pi\omega(t)}{k(1-t)}+\frac{4\omega(t)-3k}{(1-t)^{3}\omega(t)}\int_{\Sigma_{t}}|\nabla u|^{2}-\frac{1}{(1-t)^{2}}\int_{\Sigma_{t}}|\nabla u|H
\end{equation}
is monotone nondecreasing in $ t $. Furthermore, we have 
\begin{equation}\label{lim F}
\mathcal{F}_{k}(t)\leq 8\pi(\mathfrak{m}\mathfrak{c}_{\Sigma}^{-1}-\frac{2(k-1)}{k}).
\end{equation}
In particular, at $ \Sigma $,
\begin{equation}\label{F 0}
\frac{4\pi}{k}+(4-3k)\int_{\Sigma}|\nabla u|^{2}-\int_{\Sigma}|\nabla u|H\leq 8\pi(\mathfrak{m}\mathfrak{c}_{\Sigma}^{-1}-\frac{2(k-1)}{k}),
\end{equation}
and equality holds if and only if  $ (M,g) $ is isometric to $ (M_{m,r_{0}},g_{m}) $ of mass $ \mathfrak{m}=2r_{0}(k-1) $ with $ r_{0}=\frac{\mathfrak{c}_{\Sigma}}{k} $.
\end{proposition}
\begin{proof}
Observe that 
$$
\mathcal{F}_{k}(t)=-\frac{\omega(t)^{2}}{(1-t)^{2}}\mathcal{D}_{k}(t)-\frac{2k}{(1-t)^{3}\omega(t)}\int_{\Sigma_{t}}|\nabla u|^{2}+\frac{8\pi\omega(t)}{k(1-t)}.
$$
At a regular value $ t $, we compute 
$$
\begin{aligned}
\mathcal{F}_{k}^{\prime}(t)=&-\frac{\omega(t)^{2}}{(1-t)^{2}}\mathcal{D}_{k}^{\prime}(t)-\frac{2k\omega(t)}{(1-t)^{3}}\mathcal{D}_{k}(t)-\frac{2k(4\omega(t)-k)}{(1-t)^{4}\omega(t)^{2}}\int_{\Sigma_{t}}|\nabla u|^{2}\\
&+\frac{2k}{(1-t)^{3}\omega(t)}\int_{\Sigma_{t}}|\nabla u|H+\frac{8\pi}{(1-t)^{2}}\\
=&-\frac{\omega(t)^{2}}{(1-t)^{2}}\mathcal{D}_{k}^{\prime}(t)\\
\geq&0.
\end{aligned}
$$
In general, observe that the following relation holds
$$
\mathcal{F}_{k}(t)=\frac{4\omega(t)-k}{k}\mathcal{A}_{k}(t)-\frac{4\omega(t)-3k}{k^{2}}\mathcal{B}_{k}(t).
$$
Hence, we can apply a regularization argument to show that 
$$
\mathcal{F}_{k}(t_{2})-\mathcal{F}_{k}(t_{1})\geq0
$$
for $ t_{2}>t_{1} $.

To prove (\ref{lim F}), inspired by the idea in \cite{Ag2}, using the L'Hospital rule, we get
\begin{equation}\label{Llim F}
\begin{aligned}
\lim\limits_{t\to 1}\mathcal{F}_{k}(t)=&\lim\limits_{t\to 1}\frac{\frac{4\pi}{k}+\frac{4\omega(t)-3k}{(1-t)^{2}\omega(t)^{2}}\int_{\Sigma_{t}}|\nabla u|^{2}-\frac{1}{(1-t)\omega(t)}\int_{\Sigma_{t}}|\nabla u|H}{\frac{1-t}{\omega(t)}}\\
=&\lim\limits_{t\to 1}\frac{\omega(t)}{k(1-t)}\Big[-\frac{12\omega(t)^{2}-16k\omega(t)+6k^{2}}{(1-t)^{2}\omega(t)^{2}}\int_{\Sigma_{t}}|\nabla u|^{2}+\frac{6\omega(t)-4k}{(1-t)\omega(t)}\int_{\Sigma_{t}}|\nabla u|H\\
&+\Big(\int_{\Sigma_{t}}|\nabla u|H\Big)^{\prime} \Big]\\
\leq&\lim\limits_{t\to 1}\frac{\omega(t)}{k(1-t)}\Big[-\frac{12\omega(t)^{2}-16k\omega(t)+6k^{2}}{(1-t)^{2}\omega(t)^{2}}\int_{\Sigma_{t}}|\nabla u|^{2}+\frac{6\omega(t)-4k}{(1-t)\omega(t)}\int_{\Sigma_{t}}|\nabla u|H\\
&+4\pi-\frac{3}{4}\int_{\Sigma_{t}}H^{2}\Big]\\
=&\lim\limits_{t\to 1}\frac{\omega(t)}{k(1-t)}\Big[-\int_{\Sigma_{t}}\frac{1}{2}\Big(H-\frac{2(2\omega-k)}{(1-u)\omega}|\nabla u|\Big)^{2}+\frac{2\omega(t)-2k}{(1-t)\omega(t)}\int_{\Sigma_{t}}|\nabla u|H\\
&-\frac{4(k-\omega(t))^{2}}{(1-t)^{2}\omega(t)^{2}}\int_{\Sigma_{t}}|\nabla u|^{2}+4\pi-\frac{1}{4}\int_{\Sigma_{t}}H^{2} \Big].
\end{aligned}
\end{equation}
Since $ k-\omega(t)=(k-1)(1-t) $, it follows from (\ref{Llim F}), Lemma~\ref{funda u}$ (iv) $ and (\ref{lim}) that
$$
\begin{aligned}
\lim\limits_{t\to 1}\mathcal{F}_{k}(t)&\leq \lim\limits_{t\to 1} \frac{\omega(t)}{k(1-t)}\Big[-\frac{2(k-1)}{\omega(t)}\int_{\Sigma_{t}}|\nabla u|H-\frac{4(k-1)^{2}}{\omega(t)^{2}}\int_{\Sigma_{t}}|\nabla u|^{2}+4\pi-\frac{1}{4}\int_{\Sigma_{t}}H^{2}  \Big]\\
&=\lim\limits_{t\to 1}\Big[-\frac{2(k-1)}{k(1-t)}\int_{\Sigma_{t}}|\nabla u|H+\frac{\omega(t)}{k(1-t)}\Big(4\pi-\frac{1}{4}\int_{\Sigma_{t}}H^{2} \Big) \Big]\\
&\leq 8\pi(\mathfrak{m}\mathfrak{c}_{\Sigma}^{-1}-\frac{2(k-1)}{k}).
\end{aligned}
$$
\end{proof}
\begin{remark}
Comparing (\ref{0 cp}), (\ref{lim B0}) and (\ref{F 0}), one has (\ref{0 cp})$+$(\ref{lim B0}) $\Rightarrow$ (\ref{F 0}).
\end{remark}
Now, we apply the monotone quantity $ \mathcal{F}_{k}(t) $ to derive the monotone quantity $\mathcal{G}_{k}(t)$.

\textbf{Proof of Theorem 1.2}.
At a regular value $ t $, we calculate
\begin{equation}\label{G derF}
\begin{aligned}
\mathcal{G}_{k}^{\prime}(t)&=\frac{4\pi}{k\omega(t)^{2}}+\frac{4\omega(t)-3k}{(1-t)^{2}\omega(t)^{4}}\int_{\Sigma_{t}}|\nabla u|^{2}-\frac{1}{(1-t)\omega(t)^{3}}\int_{\Sigma_{t}}|\nabla u|H\\
&=\frac{1-t}{\omega(t)^{3}}\mathcal{F}_{k}(t).
\end{aligned}
\end{equation}
In general, note that 
$$
\mathcal{G}_{k}(t)=-\frac{(1-t)^{2}}{\omega(t)^{2}}\mathcal{B}_{k}(t).
$$
Thus, we can apply a regularization argument to show that 
$$
\mathcal{G}_{k}(t_{2})-\mathcal{G}_{k}(t_{1})=\int_{t_{1}}^{t_{2}}\frac{1-t}{\omega(t)^{3}}\mathcal{F}_{k}(t)dt.
$$
for $ t_{2}>t_{1} $.
By Proposition~\ref{F mono}, $ \mathcal{F}_{k}(t) $ is monotone nondecreasing in $ t $, we get
\begin{equation}\label{Gk deriv}
\mathcal{G}_{k}^{\prime}(t)=\frac{1-t}{\omega(t)^{3}}\mathcal{F}_{k}(t)\geq \frac{1-t}{\omega(t)^{3}}\mathcal{F}_{k}(0).
\end{equation}
Integrating (\ref{Gk deriv}) from $ 0 $ to $ t $ gives 
\begin{equation}\label{G-F}
\mathcal{G}_{k}(t)-\mathcal{G}_{k}(0)\geq \Big[\frac{k-2}{2(k-1)^{2}}+\frac{2\omega(t)-k}{2(k-1)^{2}\omega(t)^{2}}\Big]\mathcal{F}(0).
\end{equation}
On the other hand, by Lemma~\ref{funda u} $ (iv) $, we have 
$$ 
\lim\limits_{t\to 1}\mathcal{G}_{k}(t)=0 
$$. 
Taking $ t \to 1 $ in (\ref{G-F}) gives
\begin{equation}\label{GF1}
-\mathcal{G}_{k}(0)\geq \frac{1}{2k}\mathcal{F}_{k}(0).
\end{equation}
Now, the condition (\ref{H condition0}) implies that 
$$
\int_{\Sigma}H|\nabla u|\leq 4\pi \alpha \mathfrak{c}_{\Sigma}-k^{2}\alpha \mathfrak{c}_{\Sigma}\int_{\Sigma}|\nabla u|^{2}.
$$
Then we get 
\begin{equation}\label{GF2}
\begin{aligned}
\mathcal{F}_{k}(0)&=\frac{4\pi}{k}+(4-3k)\int_{\Sigma}|\nabla u|^{2}-\int_{\Sigma}|\nabla u|H\\
&\geq \frac{4\pi}{k^{2}}(k-k^{2}\alpha \mathfrak{c}_{\Sigma})-(3k-4-k^{2}\alpha \mathfrak{c}_{\Sigma})\int_{\Sigma}|\nabla u|^{2}\\
&\geq \frac{4\pi}{k^{2}}(k-k^{2}\alpha \mathfrak{c}_{\Sigma})-(k-k^{2}\alpha \mathfrak{c}_{\Sigma})\int_{\Sigma}|\nabla u|^{2}\\
&=-k(1-k\alpha\mathfrak{c}_{\Sigma})\mathcal{G}_{k}(0),
\end{aligned}
\end{equation}
where we have used the fact that $ 3k-4\leq k $ for $ 0<k\leq 2 $, and 
$$
\mathcal{G}_{k}(0)=-\frac{4\pi}{k^{2}}+\int_{\Sigma}|\nabla u|^{2}.
$$
It follows from (\ref{GF1}) and (\ref{GF2}) that 
\begin{equation}\label{GF3}
-\frac{1}{2}(1+k\alpha \mathfrak{c}_{\Sigma})\mathcal{G}_{k}(0)\geq 0.
\end{equation}
Since $ \alpha\in (-\frac{1}{k\mathfrak{c}_{\Sigma}},\frac{1}{k\mathfrak{c}_{\Sigma}}]  $, (\ref{GF3}) implies that 
$$
\mathcal{G}_{k}(0)\leq 0,
$$
and so 
$$
\mathcal{F}_{k}(0)\geq 0
$$
by (\ref{GF2}). Hence, from (\ref{Gk deriv}) we conclude that 
$$
\mathcal{G}_{k}^{\prime}(t)\geq 0.
$$
This means that 
$$
\mathcal{G}_{k}(t)\leq \lim\limits_{t\to 1}\mathcal{G}_{k}(t)=0.
$$
If $ \mathcal{G}_{k}(0)=0$, then $ \mathcal{G}_{k}(t)\equiv0 $ on $ [0,1) $. Thus $ \mathcal{F}_{k}(t)\equiv0 $ on $ [0,1) $ by (\ref{G derF}). Now the rigidity case follows from Proposition~\ref{F mono}.
$\hfill\square$

\section{\bf Applications}
In this section, we give some applications in the previous sections. First we obtain the mass-capacity inequality (\ref{willmore}), which extend the work in Hirsch-Miao-Tam \cite{Hrisch1} under the harmonic functions setting. Then, we generalize the boundary mean curvature condition in Oronzio \cite{Oronzio} to obtain the Bray-Miao-type mass-capacity inequalities. 

\begin{proposition}\label{cor1}
Let $ (M,g) $ be a complete, orientable, one-ended asymptotically flat 3-manifold with compact, connnected boundary $ \Sigma $ and with $ R_{g}\geq 0 $. Suppose that $ H_{2}(M,\Sigma)=0 $. Let $ u $ be the solution of (\ref{u}) and let $ 0<k<4 $. Then
	\begin{equation}\label{willmore}
	2-\frac{1}{k}\leq \mathfrak{m}\mathfrak{c}_{\Sigma}^{-1}+\frac{(\sqrt{kW}+\sqrt{kW+4-k})^{2}}{(4-k)^{2}},
	\end{equation}
	where $ W=\frac{1}{16\pi}\int_{\Sigma}H^{2} $, and equality holds if and only if $ (M,g) $ is isometric to $ (M_{m,r_{0}},g_{m}) $ of mass $ \mathfrak{m}=2r_{0}(k-1) $ with $ r_{0}=\frac{\mathfrak{c}_{\Sigma}}{k} $.
\end{proposition}
\begin{proof}
	By (\ref{0 cp}) and H\"older's inequality, 
	$$
	k(4-k)\int_{\Sigma}|\nabla u|^{2}\leq 4\pi+k\Big(\int_{\Sigma}|\nabla u|^{2}\Big)^{\frac{1}{2}}\Big(\int_{\Sigma}H^{2}\Big)^{\frac{1}{2}}.
	$$
	Let $ z=\Big(\int_{\Sigma}|\nabla u|^{2}\Big)^{\frac{1}{2}} $. Then we get
	$$
	k(4-k)z^{2}-k\sqrt{16\pi W}z-4\pi\leq 0.
	$$
	Solving this equation, we obtain
\begin{equation}\label{z1}
	z^{2}\leq \frac{4\pi(\sqrt{kW}+\sqrt{kW+4-k})^{2}}{k(4-k)^{2}}.
\end{equation}
	On the other hand, by (\ref{lim B0}), we have 
\begin{equation}\label{z2}
	2-\frac{1}{k}\leq \mathfrak{m}\mathfrak{c}_{\Sigma}^{-1}+\frac{k}{4\pi}z^{2}.
\end{equation}
Now, (\ref{willmore}) follows from (\ref{z1}) and (\ref{z2}). The equality case follows from the equality case in Proposition~\ref{pro}.
\end{proof}

\begin{corollary}
Let $ (M,g) $ be a complete, orientable, one-ended asymptotically flat 3-manifold with compact, connnected boundary $ \Sigma $ and with $ R_{g}\geq 0 $. Suppose that $ H_{2}(M,\Sigma)=0 $. Let $ u $ be the solution of (\ref{u}) and let $ 0<k<4 $. If 
$$
\frac{1}{16\pi}\int_{\Sigma}H^{2} \leq \frac{(k-2)^{2}}{k^{2}},
$$
then we have 
$$
\mathfrak{m}\geq \frac{2(k-1)\mathfrak{c}_{\Sigma}}{k},
$$
and equality holds if and only if $ (M,g) $ is isometric to $ (M_{m,r_{0}},g_{m}) $ of mass $ \mathfrak{m}=2r_{0}(k-1) $ with $ r_{0}=\frac{\mathfrak{c}_{\Sigma}}{k} $.
\end{corollary}

\textbf{Proof of Theorem 1.3}. By the condition (\ref{condi H0}), we have 
\begin{equation}\label{int H }
\int_{\Sigma}|\nabla u|H\leq 4\pi \alpha \mathfrak{c}_{\Sigma}-k^{2}\alpha \mathfrak{c}_{\Sigma}\int_{\Sigma}|\nabla u|^{2}.
\end{equation}
It follows from (\ref{0 cp}) and (\ref{condi H0}) that 
$$
\begin{aligned}
k(4-k)\int_{\Sigma}|\nabla u|^{2}&\leq 4\pi+k\int_{\Sigma}|\nabla u|H\\
&\leq 4\pi+4\pi k\alpha \mathfrak{c}_{\Sigma}-k^{3}\alpha \mathfrak{c}_{\Sigma}\int_{\Sigma}|\nabla u|^{2}.
\end{aligned}
$$
Hence we get 
$$
k(4-k+k^{2}\alpha \mathfrak{c}_{\Sigma})\int_{\Sigma}|\nabla u|^{2}\leq 4\pi (1+k\alpha \mathfrak{c}_{\Sigma}).
$$
Notice that $ k\leq 4-k $ due to $ k\leq 2 $, we get  
$$
k^{2}(1+k\alpha \mathfrak{c}_{\Sigma})\int_{\Sigma}|\nabla u|^{2}\leq 4\pi (1+k\alpha \mathfrak{c}_{\Sigma}).
$$
Since $ \alpha>-\frac{1}{k\mathfrak{c}_{\Sigma}} $, we deduce that 
\begin{equation}\label{grad pi}
\int_{\Sigma}|\nabla u|^{2}\leq \frac{4\pi}{k^{2}}.
\end{equation}
Hence, it follows from (\ref{lim B0}) that
$$
0\leq 4\pi-k^{2}\int_{\Sigma}|\nabla u|^{2}\leq 4\pi k (\mathfrak{m}\mathfrak{c}_{\Sigma}^{-1}-\frac{2(k-1)}{k}),
$$ 
Thus $ \mathfrak{m}\geq \frac{2(k-1)\mathfrak{c}_{\Sigma}}{k} $. The rigidity case follows from that of Proposition~\ref{pro}. To prove the second inequality,  we apply H\"older's inequality and (\ref{grad pi}) to get
$$
4\pi \mathfrak{c}_{\Sigma}=\int_{\Sigma}|\nabla u|\leq \Big(\int_{\Sigma}|\nabla u|^{2}\Big)^{\frac{1}{2}}|\Sigma|^{\frac{1}{2}}\leq \frac{(4\pi)^{\frac{1}{2}}}{k}|\Sigma|^{\frac{1}{2}},
$$
which implies that 
$$
\sqrt{\frac{|\Sigma|}{4\pi k^{2}}}\geq \mathfrak{c}_{\Sigma}.
$$ 
Now the rigidity case again follows from that of Proposition~\ref{pro}.
$\hfill\square$

\section{\bf Integral Identities }
In this section, we prove Theorem~\ref{th 1.6}. We establish the mass identites via a harmonic
function $ u $ that equals $ 0 $ at the boundary and is asymptotic to $ 1 $ at $ \infty $. To proceed, we introduce the $ (0,2) $-symmetric tensors of $ M $.
$$
T=\nabla^{2}u+\frac{3(2\omega-k)}{(1-u)\omega}du\otimes du,\,\,\,B=T-\frac{tr(T)}{3}g=\nabla^{2}u-\Psi_{u},
$$
where
$$
\Psi_{u}=\frac{(2\omega-k)|\nabla u|^{2}}{(1-u)\omega}g-\frac{3(2\omega-k)}{(1-u)\omega}du\otimes du.
$$
\textbf{Proof of Theorem 1.6}.
Observe that 
$$
\begin{aligned}
&B(\nu,\nu)=-H|\nabla u|+\frac{2(2\omega-k)|\nabla u|^{2}}{(1-u)\omega},\\
&B(\nu,\cdot)|_{\Sigma_{t}}=\langle \nabla^{\Sigma_{t}}|\nabla u|, \cdot \rangle,\\
&B(\cdot,\cdot)|_{\Sigma_{t}}=|\nabla u|\Big(\mathbb{II}-\frac{(2\omega-k)|\nabla u|^{2}}{(1-u)\omega}\gamma_{t} \Big),
\end{aligned}
$$
where $ \gamma_{t} $ is the induced metric on $ \Sigma_{t} $. Hence, we have
\begin{equation}\label{Phi u}
|\nabla u|^{-2}|\nabla^{2}u-\Psi_{u}|^{2}=2|\nabla u|^{-2}|\nabla^{\Sigma_{t}}|\nabla u| |^{2}+|\overset{\circ}{\mathbb{II}}|^{2}+\frac{3}{2}\Big(H-\frac{2(2\omega-k)}{(1-u)\omega}|\nabla u|\Big)^{2}.
\end{equation}
Recall that 
$$
\mathcal{D}_{k}(t)=\frac{4\pi(1-t)}{k\omega(t)}+\frac{1}{\omega(t)^{2}}\int_{\Sigma_{t}}|\nabla u|H-\frac{4\omega(t)-k}{(1-t)\omega(t)^{3}}\int_{\Sigma_{t}}|\nabla u|^{2}.
$$
By (\ref{Psi geq0}) and (\ref{Phi u}), we get
$$
\mathcal{D}_{k}(t)\geq \int_{t}^{1}\frac{1}{\omega(s)^{2}}\int_{\Sigma_{s}}\frac{1}{2}\Big[\frac{|\nabla^{2}u-\Psi_{u}|^{2}}{|\nabla u|^{2}}+R_{g} \Big]ds=\int_{t}^{1}\psi_{k}(s)ds,
$$
where 
$$
\psi_{k}(t)=\frac{1}{\omega(t)^{2}}\int_{\Sigma_{t}}\frac{1}{2}\Big[\frac{|\nabla^{2}u-\Psi_{u}|^{2}}{|\nabla u|^{2}}+R_{g} \Big].
$$
Hence, given two regular values $ t_{1}<t_{2} $, taking $ f(u)=\frac{k^{2}}{(1-u)^{3}\omega(u)} $ in Lemma~\ref{f(u)1}, one has 
\begin{equation}\label{mono B}
\begin{aligned}
\mathcal{B}_{k}(t_{2})-\mathcal{B}_{k}(t_{1})&=\int_{t_{1}}^{t_{2}}\frac{k^{2}}{(1-t)^{2}}\Big[\frac{4\pi}{k}-\frac{4\omega(t)-k}{(1-t)^{2}\omega(t)^{2}}\int_{\Sigma_{t}}|\nabla u|^{2}+\frac{1}{(1-t)\omega(t)}\int_{\Sigma_{t}}|\nabla u|H \Big]dt\\
&=k^{2}\int_{t_{1}}^{t_{2}}\frac{\omega(t)}{(1-t)^{3}}\mathcal{D}_{k}(t)dt\\
&\geq k^{2}\int_{t_{1}}^{t_{2}}\frac{\omega(t)}{(1-t)^{3}}\Big(\int_{t}^{1}\psi_{k}(s)ds\Big)dt.
\end{aligned}
\end{equation}
Taking $ t_{2}\to 1 $ and $ t_{1}\to 0 $ in (\ref{mono B}), we obtain
$$
4\pi k\Big(\mathfrak{m}\mathfrak{c}_{\Sigma}^{-1}-\frac{2(k-1)}{k}\Big)-\mathcal{B}_{k}(0)\geq k^{2}\int_{0}^{1}\frac{\omega(t)}{(1-t)^{3}}\Big(\int_{t}^{1}\psi_{k}(s)ds\Big)dt.
$$ 
Integrating by parts gives
$$
\begin{aligned}
\int_{0}^{1}\frac{\omega(t)}{(1-t)^{3}}\Big(\int_{t}^{1}\psi_{k}(s)ds\Big)dt=&\Big[\frac{2\omega(t)-k}{2(1-t)^{2}}\int_{t}^{1}\psi_{k}(s)ds\Big]\Big|_{t=0}^{t=1}+\int_{0}^{1}\frac{2\omega(t)-k}{2(1-t)^{2}}\psi_{k}(t)dt\\
=&\lim\limits_{t\to 1}\frac{2\omega(t)-k}{2(1-t)^{2}}\int_{t}^{1}\psi_{k}(s)ds-\frac{2-k}{2}\int_{0}^{1}\psi_{k}(s)ds\\
&+\int_{0}^{1}\frac{2\omega(t)-k}{2(1-t)^{2}}\psi_{k}(t)dt.
\end{aligned}
$$
We claim that 
\begin{equation}\label{cla}
\lim\limits_{t\to 1}\frac{2\omega(t)-k}{2(1-t)^{2}}\int_{t}^{1}\psi_{k}(s)ds=0.
\end{equation}
Indeed, by Lemma~\ref{funda u}, we have the following
$$
\int_{t}^{1}\int_{\Sigma_{s}}|\nabla u|^{-2}|\nabla^{\Sigma_{s}}|\nabla u||^{2}+\frac{1}{2}|\overset{\circ}{\mathbb{II}}|^{2}=O((1-t)^{1+2\tau})
$$
and 
$$
\int_{t}^{1}\int_{\Sigma_{s}}|R_{g}|=o((1-t)^{2}).
$$
Moreover, we have
$$
\int_{t}^{1}\int_{\Sigma_{s}}\Big(H-\frac{2(2\omega-k)}{(1-u)\omega}|\nabla u|\Big)^{2}=O((1-t)^{1+2\tau}).
$$
Hence, (\ref{cla}) holds.

Now, we find that 
$$
\begin{aligned}
4\pi k\Big(\mathfrak{m}\mathfrak{c}_{\Sigma}^{-1}-\frac{2(k-1)}{k}\Big)-\mathcal{B}_{k}(0)&\geq \frac{k^{2}}{2}\int_{0}^{1}\Big[\frac{2\omega(t)-k}{(1-t)^{2}}+k-2\Big]\psi_{k}(t)dt\\
&=\frac{k^{2}}{4}\int_{0}^{1}\frac{1}{\omega(t)^{2}}\Big[\frac{2\omega(t)-k}{(1-t)^{2}}+k-2\Big]\int_{\Sigma_{t}}\Big[\frac{|\nabla^{2}u-\Phi_{u}|^{2}}{|\nabla u|^{2}}+R_{g} \Big]dt\\
&=\frac{k^{2}}{4}\int_{M}\frac{1}{\omega^{2}}\Big[\frac{2\omega-k}{(1-u)^{2}}+k-2\Big]\Big[\frac{|\nabla^{2}u-\Phi_{u}|^{2}}{|\nabla u|}+R_{g}|\nabla u|\Big]. 
\end{aligned}
$$
This proves (\ref{int 1}).

By Proposition~\ref{F mono}, $ \mathcal{F}_{k}(t) $ is is monotone nondecreasing in $ t $.  Similarly, we have 
$$
\mathcal{F}_{k}(t_{2})-\mathcal{F}_{k}(t_{1})\geq \int_{t_{1}}^{t_{2}}\Big(\frac{\omega(t)}{1-t}\Big)^{2}\psi_{k}(t)dt.
$$
Letting $ t_{2}\to 1 $ and $ t_{1}\to 0 $, we get 
$$
\begin{aligned}
8\pi(\mathfrak{m}\mathfrak{c}_{\Sigma}^{-1}-\frac{2(k-1)}{k})-\mathcal{F}_{k}(0)&\geq \int_{0}^{1}\Big(\frac{\omega(t)}{1-t}\Big)^{2}\psi_{k}(t)dt\\
&=\int_{0}^{1}\frac{1}{2(1-t)^{2}}\int_{\Sigma_{t}}\Big[\frac{|\nabla^{2}u-\Psi_{u}|^{2}}{|\nabla u|^{2}}+R_{g} \Big]\\
&=\int_{M}\frac{1}{2(1-u)^{2}}\Big[\frac{|\nabla^{2}u-\Psi_{u}|^{2}}{|\nabla u|}+R_{g}|\nabla u|\Big].
\end{aligned}
$$
This proves (\ref{int 2}).
$\hfill\square$

\end{document}